\documentclass[a4paper, 12pt]{article}

\usepackage[T1]{fontenc}
\usepackage[utf8]{inputenc}
\usepackage{amsmath,amssymb,amsthm}
\allowdisplaybreaks
\usepackage{a4wide}
\usepackage{mathrsfs}
\usepackage{bbold}
\usepackage{graphicx}

\usepackage{enumitem}

\usepackage[sort]{natbib}
\makeatletter
\def\NAT@spacechar{~}
\newcommand{\myciteyearp}[1]{%
  \begingroup
    \def\NAT@nmfmt##1{}%
    \let\saved@NAT@open\NAT@open
    \def\NAT@open{\unskip\saved@NAT@open}%
    \citet{#1}%
  \endgroup
}
\makeatother

\usepackage{authblk}

\newtheorem{theorem}{Theorem}
\newtheorem{definition}[theorem]{Definition}
\newtheorem{lemma}[theorem]{Lemma}
\newtheorem{proposition}[theorem]{Proposition}
\newtheoremstyle{myremark}{}{}{}{}{\bfseries}{.}{ }{}
\theoremstyle{myremark}
\newtheorem{remark}[theorem]{Remark}

\newcommand{\NN}{\mathbb{N}}
\newcommand{\RR}{\mathbb{R}}
\newcommand{\UU}{\mathbb{U}}
\newcommand{\KK}{\mathbb{K}}
\newcommand{\ZZ}{\mathbb{Z}}
\newcommand{\ind}{\mathbb{1}}
\newcommand{\EC}{\mathscr{C}}
\newcommand{\ED}{\mathscr{D}}
\newcommand{\EX}{\mathscr{X}}
\newcommand{\dd}{\,\mathrm{d}}
\newcommand{\Ex}{\mathbf{E}}
\newcommand{\Pb}{\mathbf{P}}
\newcommand{\bil}{\mathopen}
\newcommand{\bir}{\mathclose}
\newcommand{\Card}{\mathop{\mathrm{Card}}}
\newcommand{\sign}{\mathop{\mathrm{sign}}}
\newcommand{\argsup}{\mathop{\rm argsup}\limits}
\newcommand{\mylimsup}{\mathop{\overline\lim}\limits}

\begin{document}

\title{Extension of Likelihood Ratio Analysis Method to Skorokhod $M_1$
  Topology (with Application to Poissonian Smooth Change-Point Model)}
\author[1]{Arij Amiri}
\author[2]{Sergueï Dachian}
\affil[1,2]{\small Univ.\ Lille, CNRS, UMR 8524 --- Laboratoire Paul Painlevé,
  F--59000 Lille, France}
\affil[1]{\small Seenovate, F--75009 Paris, France}
\date{}
\maketitle

\begin{abstract}
We extend the Ibragimov--Khasminskii likelihood ratio analysis method, in the
one-dimensional parameter case, to a framework based on the Skorokhod $M_1$
topology.  The proposed framework applies to a broad class of statistical
models, including those for which the normalized likelihood ratio processes
are continuous while the limiting likelihood ratio process is discontinuous, a
setting outside the scope of classical approaches based on the uniform or
Skorokhod $J_1$ topologies.  We derive sufficient conditions for the weak
convergence of likelihood ratio processes in the space of càdlàg functions
on~$\RR$ vanishing at~$\pm\infty$, endowed with the Skorokhod $M_1$ topology,
and show how this convergence yields the asymptotic behavior of the maximum
likelihood estimator.  We also introduce new techniques for controlling the
$M_1$ modulus of continuity.  Although these techniques are particularly well
suited to models in which all jumps and steep continuous transitions are in
the same direction, they are of independent interest and may prove useful
beyond this setting.  As an application, we study a smooth change-point model
for inhomogeneous Poisson processes in the fast regime, where the transition
interval shrinks faster than~$1/n$.  We establish that in this case the
maximum likelihood estimator has the same asymptotic behavior (consistency,
rate of convergence, limiting distribution and convergence of moments) as in
the corresponding ``pure'' change-point model.


\end{abstract}

\section{Introduction}

The Ibragimov--Khasminskii likelihood ratio analysis method is one of the most
powerful tools for the asymptotic study of both regular and non-regular
parametric statistical models for various types of observations
(i.i.d.~observations, signal-in-noise models, Poisson processes, diffusion
processes, time series, and so~on).  Introduced and developed by
\citeauthor{IbHas81} in a series of 
works~\myciteyearp{IbHas70,IbHas72,IbHas72S,IbHas73,IbHas75,IbHas74,IbHas81},
this method is based, in the one-dimensional parameter case, on the asymptotic
analysis of the \emph{(normalized) likelihood ratio (process)}
\[
Z_n(u) = \frac{L \bigr(\theta+u\varphi_n, X^{(n)}\bigl)}{L \bigr(\theta,
  X^{(n)}\bigl)}\,,\qquad
u\in\bigl(\varphi_n^{-1}(\alpha-\theta),\varphi_n^{-1}(\beta-\theta)\bigr),
\]
where $\theta\in\Theta=(\alpha, \beta)\subset\RR$ is the true value of the
unknown parameter, $\varphi_n \searrow 0$ is a suitably chosen normalization
rate, $X^{(n)}$ is the observation, and $L$ is the likelihood of the model.
First, the convergence of $Z_n$ (suitably extended to the whole real line) to
a \emph{limiting likelihood ratio (process)\/}~$Z$ is established, and then
the asymptotic properties of the maximum likelihood estimator (MLE) and of the
Bayesian estimators (BEs) are deduced.  In particular, this approach allows
one to obtain consistency, rates of convergence, limiting distributions and
convergence of moments of the estimators, as well as asymptotic efficiency of
the BEs.  Note that in some cases (in particular in regular models), the MLE
is also asymptotically efficient, and that the weak convergence of the
likelihood ratio in an appropriate function space is essential for its study,
while for the BEs, the convergence of finite-dimensional distributions
(together with some inequalities) is sufficient.

The choice of the function space in which the convergence of the likelihood
ratio is studied is a crucial part of the method.  In the regular case, as
well as in some singular problems where both the normalized and limiting
likelihood ratio processes have continuous trajectories, the convergence is
studied in the space $\EC_0^{(U)}(\RR)$ of continuous functions on~$\RR$
vanishing at~$\pm\infty$, endowed with the uniform topology (see, for example,
\citet{IbHas72S,IbHas73,IbHas74,IbHas81}, as well as
\citet{Kut77,Kut79,Kut84,Kut98,Kut23} and \citet{D03S,D11S} for Poissonian
observations).  In problems where the trajectories of the likelihood ratio
processes are discontinuous càdlàg functions (such as, for example,
change-point models for i.i.d.\ and Poissonian observations), the standard
approach is to work in the Skorokhod space $\ED_0^{(J_1)}(\RR)$ of càdlàg
functions on~$\RR$ vanishing at~$\pm\infty$, endowed with the usual Skorokhod
topology~$J_1$ (see, for example, \citet{IbHas72,IbHas81}, as well as
\citet{Kut78,Kut84,Kut98,Kut23} for Poissonian observations).  To the best of
our knowledge, these are the only two function spaces used in the likelihood
ratio analysis method so far.

However, there are statistical models for which neither of these two classical
settings is adequate.  This is the case, for example, of the \emph{smooth
change-point\/} model for Poisson processes (more details will be given below)
considered in a recent paper \citet{Amiri-Dachian} of the authors.  Indeed, in
this model the trajectories of the normalized likelihood ratio processes are
continuous (for every $n$), while those of the limiting likelihood ratio
process are discontinuous.  Such convergence cannot take place in the
topology~$J_1$, introduced by \citeauthor{Sko} in~\myciteyearp{Sko} for the
case of càdlàg functions on a compact interval (see also \citet{Billingsley}
and \citet{Whitt}), as the latter is adapted to the situations where each jump
of the limiting function is ``approximated'', as $n\to+\infty$, by a jump in
the prelimit function, and hence does not allow the convergence of continuous
functions to discontinuous ones.

The main contribution of this paper is therefore the extension of the
Ibragimov--Khasminskii likelihood ratio analysis method to a weaker Skorokhod
topology, namely the topology~$M_1$.  This topology, also introduced by
\citeauthor{Sko} in~\myciteyearp{Sko} for the case of càdlàg functions on a
compact interval, is defined through parametric representations of completed
graphs of càdlàg functions.  Unlike the topology~$J_1$, it allows a jump of
the limiting function to be approximated by a steep continuous transition in
the prelimit function.  This makes the topology~$M_1$ a natural framework for
likelihood ratio analysis in models where continuous likelihood ratio
processes converge to discontinuous limits.  General references on Skorokhod
topologies and weak convergence in spaces of càdlàg functions include
\citet{Avram-89,Avram-92,Billingsley,Pomarede} and \citet{Whitt}.

In order to develop a version of the likelihood ratio analysis method which
uses the space $\ED_0^{(M_1)}(\RR)$ of càdlàg functions on~$\RR$ vanishing
at~$\pm\infty$, endowed with the topology~$M_1$, we first recall the weak
convergence criteria in the spaces $\EC_0^{(U)}(\RR)$ and
$\ED_0^{(J_1)}(\RR)$, and introduce the corresponding criterion
in~$\ED_0^{(M_1)}(\RR)$.  We then adapt the techniques used in the likelihood
ratio analysis method to the topology~$M_1$.  In particular, we derive
sufficient conditions for the weak convergence of normalized likelihood ratio
processes in $\ED_0^{(M_1)}(\RR)$, based on convergence of finite-dimensional
distributions, control of the $M_1$ modulus of continuity on finite intervals
and exponential-type bounds on the tails of the likelihood ratio processes.

Next, we show how the asymptotic properties of the MLE can be deduced from
this $M_1$~convergence.  More precisely, if the normalized likelihood ratio
process $Z_n$ converges weakly in $\ED_0^{(M_1)}(\RR)$ to a limiting
process~$Z$, and if $Z$ almost surely attains its maximum at a unique point,
then the MLE $\hat{\theta}_n$ converges at rate $\varphi_n$ and its limiting
distribution is given by the location of this maximum:
\[
\varphi_n^{-1}(\hat{\theta}_n-\theta) \Longrightarrow \eta = \argsup_{u\in\RR}
Z(u).
\]
This extends the classical $J_1$ argument of \citet{IbHas81} to the $M_1$
framework.  Note that the uniqueness of the maximizer is essential in this
step.  We can nevertheless mention the paper~\citet{SeiSen}, which provides
some results for the case where the maximizer is not unique, albeit in a very
particular setting.

Another contribution of the paper is the development of tools for the control
of the $M_1$ modulus of continuity (required to ensure the $M_1$ tightness).
We introduce the moduli of increase and decrease and relate them to the
uniform and $M_1$ moduli of continuity.  We also introduce the notions of
Lipschitz growth and decay.  These ``one-sided'' tools make, in particular,
the topology~$M_1$ much easier to control than the topology~$J_1$ in the cases
where all jumps and steep continuous transitions are in the same direction.
In particular, for change-point models in which all jumps have the same sign,
the use of the topology~$M_1$ can essentially simplify the existing proofs by
avoiding some delicate estimates on the probability of observing several close
jumps.  Let us also note that the notions of moduli of increase and decrease
introduced here are closely related to quantities measuring monotonicity and
bounded variation of functions, as considered for instance in \citet*{Appell}
and \citet{Banas}.

As we already mentioned, the motivation for developing the $M_1$ extension of
the likelihood ratio analysis method comes from the smooth change-point model
for inhomogeneous Poisson processes considered in \citet{Amiri-Dachian}.
Inhomogeneous Poisson processes are both simple enough to allow the use of
likelihood ratio analysis, and sufficiently rich to model various random
phenomena in diverse applied fields, such as biology, telecommunications,
seismology, astronomy, reliability theory, and so on (see, for example,
\citet{ChaFin18,CoLew66,Sarkar16,SnyMil12,Streit10} and \citet{Thompson88}).
Recall that $X=\bigl(X(t),\ 0\leq t\leq T\bigr)$ is an inhomogeneous Poisson
process (on a finite interval~$[0,T]$) with intensity function~$\lambda$, if
$X(0)=0$ and the increments of~$X$ on disjoint intervals are independent
Poisson random variables, with the parameter of $X(v)-X(u)$ given by~$\int_u^v
\lambda(t) \dd t$.

In the smooth change-point model for Poisson processes, the intensity function
is supposed to change continuously from one level to another over a transition
interval of length~$\delta_n \to 0$, and the unknown parameter is the location
$\theta$ of this transition.  For each fixed~$n$, the intensity function is
continuous, but as $n\to+\infty$, the model becomes asymptotically close to a
``pure'' (discontinuous) change-point model.  Such a model is natural from an
applied point of view, because physical systems cannot switch instantaneously
from one regime to another, even though they may do so over a very short time
interval.

It was shown in \citet{Amiri-Dachian} that the asymptotic behavior of this
smooth change-point model depends on the rate at which $\delta_n$ tends to
zero.  If $\delta_n$ goes to zero slower than $1/n$ (that is,
$n\delta_n\to+\infty$, which we call \emph{slow regime\/}), the model is
locally asymptotically normal (though with a non-standard normalization rate
$\varphi_n=\sqrt{\delta_n/n}\,$), and the MLE and the BEs have asymptotically
Gaussian behavior and are asymptotically efficient.  If, on the contrary,
$\delta_n$ goes to zero faster than $1/n$ (that is, $n\delta_n \to 0$, which
we call \emph{fast regime\/}), the transition interval becomes ``negligible'':
the normalization rate~$1/n$ and the limiting likelihood ratio process are the
same as in the corresponding pure change-point model.  In
\citet{Amiri-Dachian}, the properties of the BEs were established in the fast
regime, but the MLE could not be studied by the classical likelihood ratio
analysis method, as the topology~$J_1$ was inappropriate (here the likelihood
ratio processes are continuous, while the limiting process is discontinuous).

The final contribution of the paper is the application of the $M_1$ version of
the likelihood ratio analysis method developed in this paper to the study of
the MLE in the fast regime of the Poissonian smooth change-point model.  We
show that if $n\delta_n \to 0$, the MLE $\hat{\theta}_n$ is consistent and
converges at rate $1/n$.  More precisely,
\[
n(\hat{\theta}_n-\theta) \Longrightarrow \eta_{a,b},
\]
where $\eta_{a,b}$ is the location of the maximum of the limiting likelihood
ratio process associated with the corresponding Poissonian pure change-point
model.  Convergence of polynomial moments also holds.  So, when the transition
interval shrinks sufficiently fast, both the MLE and the BEs have the same
asymptotic behavior as in the pure change-point model.  We believe that these
results provide a theoretical explanation of why pure change-point models have
proved successful in real-world applications, despite the fact that in
physical systems the intensity cannot change instantaneously from one level to
another.

Although the smooth change-point model for Poisson processes is the main
application presented here, the scope of the method is broader.  The
$M_1$-based likelihood ratio analysis method can be applied to other models in
which continuous normalized likelihood ratios converge to discontinuous
limits.  It may also prove useful in models with discontinuous likelihood
ratios for which the classical ($J_1$-based) method already applies, but where
the weaker topology~$M_1$ can lead to simpler proofs (this is, in particular,
the case for models with several pure change-points whose jumps are in the
same direction).

The paper is organized as follows.  In Section~\ref{Sec-Extension}, we recall
the likelihood ratio analysis method and extend it to the space
$\ED_0^{(M_1)}(\RR)$.  We establish weak convergence tools for likelihood
ratio processes and show that the asymptotic behavior of the MLE follows from
convergence in the topology~$M_1$.  In Section~\ref{Sec-outils}, we develop
auxiliary tools for the study of the $M_1$ modulus of continuity, such as the
moduli of increase and decrease, or Lipschitz growth and decay.  In
Section~\ref{Sec-EMV-rapide}, we apply the method to the smooth change-point
model for Poisson processes and prove the asymptotic properties of the MLE in
the fast regime.  We conclude by discussing in Section~\ref{Sec-Discussion}
some potential extensions, including the \emph{critical regime\/} $n\delta_n
\to c>0$, models with several smooth change-points and possible
simplifications of existing proofs for pure change-point models.

Let us finally note that a preliminary presentation (in French) of the results
of the present paper was included in the Ph.D.\ thesis~\citet{Amiri22} of one
of the authors.

\section{Extension of the likelihood ratio analysis method to the Skorokhod
  $M_1$ topology}
\label{Sec-Extension}

We start this section by recalling the likelihood ratio analysis method: a
powerful method for studying regular and non-regular parametric statistical
models introduced by \citeauthor{IbHas81}
in~\myciteyearp{IbHas70,IbHas72,IbHas72S,IbHas73,IbHas74,IbHas81}.

Consider a parametric model where we observe some $X^{(n)}$, $n\in\NN^*$, with
distribution~$\Pb_\theta^{(n)}$, depending on an unknown parameter (to be
estimated) $\theta\in\Theta=(\alpha, \beta)\subset\RR$, and with likelihood
\[
L \bigr(\theta, X^{(n)}\bigl) =
\frac{\dd\Pb_{\theta}\bigl(X^{(n)}\bigr)}{\dd\mu}\,,\qquad \theta\in\Theta,
\]
where $\mu$ is some reference measure (on the space of observations).

The first step of the likelihood ratio analysis method is the study of the
asymptotic (as~$n\to\infty$) behavior under $\Pb_{\theta}^{(n)}$ of the
\emph{(normalized) likelihood ratio (process)}
\[
Z_n(u) = Z_n^{(\theta)}(u) = \frac{L \bigr(\theta+u\varphi_n, X^{(n)}\bigl)}{L
  \bigr(\theta, X^{(n)}\bigl)}\,,\qquad u\in\UU_n,
\]
where
$\UU_n=\bigl(\varphi_n^{-1}(\alpha-\theta),\varphi_n^{-1}(\beta-\theta)\bigr)$
and $\varphi_n$ is some sequence decreasing to $0$ called \emph{likelihood
normalization rate}.  Then, the asymptotic properties (such as consistency,
rates of convergence and limiting distributions, convergence of moments and
asymptotic efficiency) of the maximum likelihood estimator (MLE) and of the
Bayesian estimators (BEs) are deduced from this behavior.

A very important ingredient of the method is the weak convergence (with the
right choice of $\varphi_n$) of the process $Z_n$ to a non-degenerate limiting
process~$Z=Z^{(\theta)}$ called \emph{limiting likelihood ratio (process)}.
Note that usually the trajectories of these processes are either continuous or
càdlàg (that is, right-continuous and having left limits everywhere).  Note
also that, since~$\UU_n\uparrow\RR$, the limiting likelihood ratio $Z$ must be
defined on the whole real line.  Therefore, the normalized likelihood ratios
$Z_n$, $n\in\NN^*$, also need to be extended to~$\RR$.  As the values of $u$
such that $\bil|u\bir|$ is large correspond to the parameter values
$\theta+u\varphi_n$ far from the true value~$\theta$, it is natural for
$Z_n(u)$ to be small for these values of~$u$.  So the process~$Z_n$ is
extended to~$\RR$ in such a way that its trajectories decrease to~$0$ when
$u\to\infty$ (staying continuous or càdlàg as appropriate).

To the best of our knowledge, the likelihood ratio analysis was developed (and
applied) using only two function spaces for the weak convergence of the
likelihood ratio.  In the case where the trajectories of the likelihood ratio
processes are continuous, the weak convergence was considered in the space
$\EC_0(\RR)$ of continuous functions on $\RR$ vanishing at~$\pm\infty$,
endowed with the uniform topology~$U$.  And in the case where the trajectories
of the likelihood ratio processes are discontinuous càdlàg functions, the weak
convergence was considered in the space $\ED_0(\RR)$ of càdlàg functions
on~$\RR$ vanishing at~$\pm\infty$, endowed with the usual Skorokhod
topology~$J_1$.

However, in some cases (namely, when the trajectories of the likelihood ratio
processes are continuous, while the trajectories of the limiting likelihood
ratio process are discontinuous càdlàg functions), the weak convergence cannot
hold in either of the above two spaces.  For this reason, we adapt the
likelihood ratio analysis method to a different topology, the Skorokhod $M_1$
topology.  This topology was introduced (for the case of càdlàg functions on a
compact interval) by \citeauthor{Sko} in his seminal paper~\myciteyearp{Sko},
where he also introduced the topology $J_1$, as well as two other related
topologies ($J_2$~and~$M_2$) that we do not discuss in this paper.

\subsection{Weak convergence in the spaces $\EC_0^{(U)}(\RR)$,
  $\ED_0^{(J_1)}(\RR)$ and $\ED_0^{(M_1)}(\RR)$}

Let us recall more precisely the spaces $\EC_0^{(U)}(\RR)$
and~$\ED_0^{(J_1)}(\RR)$, as well as introduce the space~$\ED_0^{(M_1)}(\RR)$.
Here and in the sequel, when a space $\EX$ is endowed with a topology $S$, we
use the notation $\EX^{(S)}$.

The space $\EC_0^{(U)}(\RR)$ was introduced by Ibragimov and Khasminskii and
allows to deal with the case where the trajectories of both the normalized
likelihood ratios and the limiting likelihood ratio are continuous (this is
the case, for example, in regular models or in models with cusp).  It is a
direct adaptation of the space $\EC^{(U)}\bigl([a,b]\bigr)$ of continuous
functions on a compact interval (for more details on this space we can refer
to \citet{Billingsley}) to the case of continuous functions on $\RR$ vanishing
at~$\pm\infty$.

\begin{definition}
We denote\/ $\EC_0(\RR)$ the space of continuous functions on\/ $\RR$
vanishing at\/ $\pm \infty$ and we endow it with the uniform topology\/ $U$,
induced by the uniform distance
\[
d^{(U)}(f,g) = \lVert f-g \rVert_\infty = \sup_{u\in \RR} \bil| f(u)-g(u)
\bir|, \qquad f,g\in \EC_0(\RR).
\]
\end{definition}

The modulus of continuity associated with the topology $U$ is
\[
\Delta_h^{(U)}(f) = \sup_{u,v\in \RR \::\:\bil|u-v\bir|\leq h}
\bil|f(u)-f(v)\bir| + \sup_{|u|>1/h} \bil|f(u)\bir|,
\]
where $f\in \EC_0(\RR)$ and $h>0$, and a weak convergence criterion for
processes with trajectories in $\EC_0^{(U)}(\RR)$ is given by the following
Prokhorov-type theorem.

\begin{theorem}
Let\/ $Y_n$, $n\in \NN^*$, and\/ $Y$ be stochastic process with trajectories
in\/ $\EC_0(\RR)$.  Then,\/ $Y_n$ converges weakly to\/ $Y$ in\/
$\EC_0^{(U)}(\RR)$ if and only if:
\begin{itemize}
\item the finite-dimensional distributions of\/ $Y_n$ converge to those
  of\/~$Y$;
\item for all\/ $\varepsilon>0$, we have
\[
\lim_{h\to 0} \ \mylimsup_{n\to +\infty}
\Pb\bigl(\Delta^{(U)}_h(Y_n)>\varepsilon\bigr)=0.
\]
\end{itemize}
\end{theorem}

Let us note that, as in the case of processes on a compact interval, the
second condition of the above theorem (together with the weak convergence of
$Y_n(0)$) guarantees the tightness of the family of probability measures
induced by the processes $Y_n$, $n \in \NN^*$, on the
space~$\EC_0^{(U)}(\RR)$, and hence (together with the convergence of
finite-dimensional distributions) the weak convergence of $Y_n$.

The space $\ED_0^{(J_1)}(\RR)$ was also introduced by Ibragimov and
Khasminskii and allows to deal with the case where the normalized likelihood
ratios are discontinuous (this is the case, for example, in many change-point
models).  It is a direct adaptation of the space
$\ED^{(J_1)}\bigl([a,b]\bigr)$ of càdlàg functions on a compact interval to
the case of càdlàg functions on~$\RR$ vanishing at~$\pm\infty$.  Recall that
the space $\ED^{(J_1)}\bigl([a,b]\bigr)$ was introduced by \citeauthor{Sko}
in~\myciteyearp{Sko}, but we can also refer \citet{Billingsley} and
\citet{Whitt}, where the setup is slightly different: while Skorokhod
restricts the space to functions which are left-continuous at the right edge
of the interval, they do not impose this restriction, and hence need
additional hypotheses in some of their results.  Let us also note that it is
the fact that we consider only vanishing at~$\pm\infty$ functions which allows
this adaptation.  Roughly speaking, we remain in the framework of functions
defined on a compact set (a~compactification of~$\RR$): if the functions
belonging to $\ED^{(J_1)}\bigl([a,b]\bigr)$ were assumed to be continuous at
the edges (that~is, at~$a$ and~$b$), the functions belonging to
$\ED_0^{(J_1)}(\RR)$ are still assumed to be continuous (but~fixing, in
addition, the values to~$0$) at the edges (that~is, at~$-\infty$
and~$+\infty$).

\begin{definition}
\label{Def_J_1}
We denote\/ $\ED_0(\RR)$ the space of càdlàg functions on\/ $\RR$ vanishing
at\/ $\pm \infty$ and we endow it with the topology\/ $J_1$ induced by the
distance
\[
d^{(J_1)}(f,g) = \inf \biggl\{\sup_{u\in \RR} \bigl|
f(u)-g\bigl(\lambda(u)\bigr) \bigr| + \sup_{u\in \RR} |u-\lambda(u)| \biggr\},
\qquad f,g\in \ED_0(\RR),
\]
where the\/ $\inf$ is taken over all continuous one-to-one mappings\/
$\lambda$ from\/ $\RR$ to\/~$\RR$.
\end{definition}

The modulus of continuity associated with the topology $J_1$ is
\[
\Delta_h^{(J_1)}(f) = \sup_{u,u',u''\in \RR \: : \: u-h\leq u'\leq u\leq
  u''\leq u+h} \min \bigl\{ \bil|f(u)-f(u')\bir|, \bil|f(u)-f(u'')\bir|
\bigr\} + \sup_{|u|>1/h} \bil|f(u)\bir|,
\]
where $f\in \ED_0(\RR)$ and $h>0$, and a weak convergence criterion for
processes with trajectories in $\ED^{(J_1)}_0(\RR)$ is given by the following
Prokhorov-type theorem.

\begin{theorem}
\label{TconvJ1}
Let\/ $Y_n$, $n\in \NN^*$, and\/ $Y$ be stochastic process with trajectories
in\/ $\ED_0(\RR)$.  Then,\/ $Y_n$ converges weakly to\/ $Y$ in\/
$\ED_0^{(J_1)}(\RR)$ if and only if:
\begin{itemize}
\item the finite-dimensional distributions of\/ $Y_n$ converge to those of\/
  $Y$ on a dense subset\/~$T$ of\/~$\RR$;
\item for all\/ $\varepsilon>0$, we have
\[
\lim_{h\to 0} \ \mylimsup_{n\to +\infty}
\Pb\bigl(\Delta^{(J_1)}_h(Y_n)>\varepsilon\bigr)=0.
\]
\end{itemize}
\end{theorem}

As we have already said, up to the best of our knowledge, until now the
likelihood ratio analysis method was only developed and applied using either
the space $\EC_0^{(U)}(\RR)$, or the space $\ED_0^{(J_1)}(\RR)$.  The space
$\EC_0^{(U)}(\RR)$ was used, for example, for regular and cusp i.i.d.\ models
by \citeauthor{IbHas81} in~\myciteyearp{IbHas70, IbHas72S, IbHas73, IbHas74,
  IbHas81}, for regular Poissonian models by \citeauthor{Kut98}
in~\myciteyearp{Kut77,Kut79,Kut84,Kut98}, and for cusp Poissonian models by
\citeauthor{D03S} in~\myciteyearp{D03S}.  As to the space
$\ED^{(J_1)}_0(\RR)$, it has been used, for example, for i.i.d.\ change-point
models by \citeauthor{IbHas81} in~\myciteyearp{IbHas70,IbHas72,IbHas81}, and
for Poissonian change-point models by \citeauthor{Kut98}
in~\myciteyearp{Kut78,Kut84,Kut98}.  Note that these spaces have also been
used for a multitude of other observation models.

However, neither of these spaces can be used in a situation where the
trajectories of the likelihood ratio processes $Z_n$, $n\in\NN^*$, are
continuous, while the ones of the limiting likelihood ratio process $Z$ are
discontinuous (for example, the fast regime of the Poissonian smooth
change-point model already mentioned in the introduction and studied below in
Section~\ref{Sec-EMV-rapide}).  Indeed, $\EC^{(U)}_0(\RR)$ cannot be used
since the trajectories of~$Z$ do not belong to this space, while
$\ED^{(J_1)}_0(\RR)$ cannot be used since the topology~$J_1$ does not allow
the convergence of continuous functions to discontinuous limits.  To overcome
this problem, we endow the space $\ED_0(\RR)$ with the topology $M_1$ which,
in contrary to $J_1$, allows such convergence.

So, below we introduce the space $\ED^{(M_1)}_0(\RR)$, which is a direct
adaptation of the space $\ED^{(M_1)}\bigl([a,b]\bigr)$ of càdlàg functions on
a compact interval (for more details on this space we refer to \citet{Sko} and
\citet{Whitt}) to the case of càdlàg functions on $\RR$ vanishing
at~$\pm\infty$.  As in the case of functions on a compact interval, in order
to define the topology~$M_1$, we need the notions of the graph of a function
$f\in\ED_0(\RR)$ and of its parametric representation.
\begin{definition}
Let\/ $f\in\ED_0(\RR)$.
\begin{enumerate}
\item The graph of\/ $f$, noted\/ $\Gamma_f$, is the closed subset of\/ $\RR^2$
  given by
\[
\Gamma_f = \bigl\{ (u,v)\in \RR^2 \: ; \quad u \in \RR,\ v\in [ f(u-),f(u)
  ] \bigr\}.
\]
\item Let\/ $(u_1,v_1),(u_2,v_2)\in \Gamma_f$.  We say that\/ $(u_1,v_1)\leq
  (u_2,v_2)$ if either we have\/ $u_1<u_2$, or we have\/ $u_1=u_2$ and
  $\bil|f(u_1-)-v_1 \bir|\leq \bil|f(u_1-)-v_2 \bir|$.
\item A parametric representation of the graph\/ $\Gamma_f$ is given by a
  pair\/ $(t,y)$ of continuous functions on\/~$\RR$, such that for all\/
  $(u,v)\in \Gamma_f$, there exists\/ $s\in \RR$ verifying\/ $u=t(s)$ and\/
  $v=y(s)$, and\/ $s\mapsto\bigl(t(s),y(s)\bigr)$ is an increasing function in
  the sense of the above introduced order, that is,\/
  $\bigl(t(s),y(s)\bigr)\leq\bigl(t(s'),y(s')\bigr)$ as soon as\/ $s\leq s'$.
\end{enumerate}
\end{definition}

We can now define the topology $M_1$ on~$\ED_0(\RR)$.

\begin{definition}
\label{Def_M_1}
The topology\/ $M_1$ on\/ $\ED_0(\RR)$ is the topology induced by the distance
\[
d^{(M_1)}(f,g)=\inf \: \: \sup_{s\in \RR} \: \: d_1\Bigl(
\bigl(t_1(s),y_1(s)\bigr) ; \bigl(t_2(s),y_2(s)\bigr) \Bigr),\qquad f,g\in
\ED_0(\RR),
\]
where the\/ $\inf$ is taken over all pairs of parametric representations\/
$(t_1,y_1)$ of\/ $\Gamma_f$ and\/ $(t_2,y_2)$ of\/~$\Gamma_g$.
\end{definition}

Here, following the original Skorokhod's definition, we use the
\emph{Manhattan distance\/} $d_1\bigl((x,y);(u,v)\bigr) = \bil|x-u\bir| +
\bil|y-v\bir|$ on~$\RR^2$, but a definition using an arbitrary distance
on~$\RR^2$ would be of course equivalent.

To illustrate the difference between the topologies $J_1$ and~$M_1$, consider
two functions $f_\delta$ and~$g$ whose graphs are represented in
Figure~\ref{fig-illust} by continuous and dashed lines, respectively, with
$\Delta > 0$ fixed and $\delta \searrow 0$.  Then $d^{(J_1)}(f_\delta,g) =
\Delta/2+\delta/2 \to \Delta/2 > 0$ (note that here the mappings $\lambda$
attaining the infimum in Definition~\ref{Def_J_1} are such that $\lambda(a) =
a+\delta/2$), while~$d^{(M_1)}(f_\delta,g) = \delta \to 0$ (note that here the
pairs of parametric representations attaining the infimum in
Definition~\ref{Def_M_1} are such that the points $A_0$ and $B_0$ correspond
to a same value of the parameter~$s$).

\begin{figure}[!ht]
\centering
\includegraphics[scale=0.35]{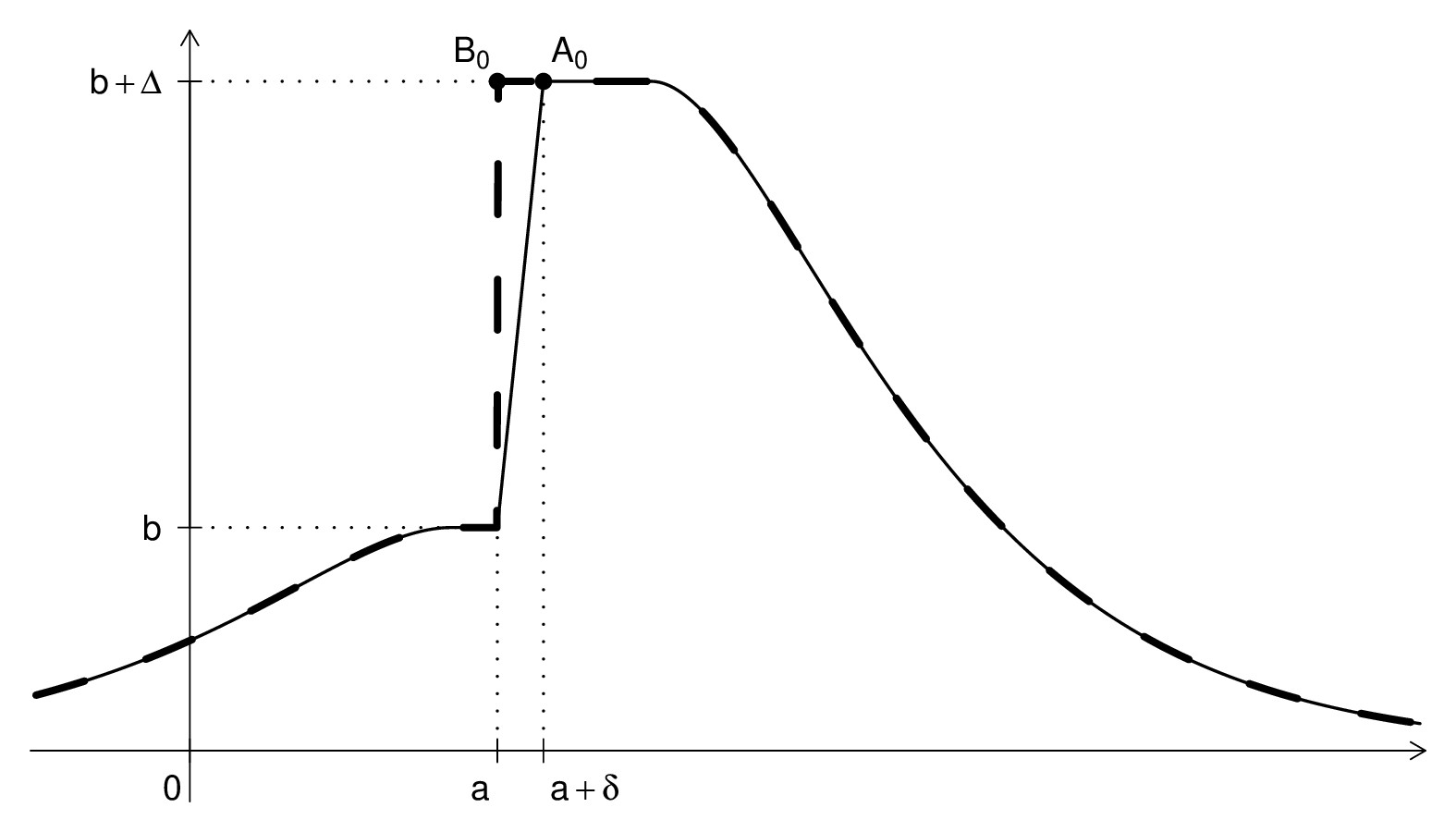}
\caption{difference between the convergences in $J_1$ and in $M_1$}
\label{fig-illust}
\end{figure}

The modulus of continuity associated with the topology $M_1$ is
\begin{equation}
\label{MC0}
\Delta_h^{(M_1)}(f)=\sup_{u,u',u''\in \RR \: : \: u-h\leq u'\leq u\leq u''\leq
  u+h} d\bigl( f(u);[f(u'), f(u'')] \bigr) + \sup_{|u|>1/h} \bil|f(u)\bir|,
\end{equation}
where $f\in \ED_0(\RR)$ and $h>0$.  Here and in the sequel, for $a,b,c \in
\RR$, we denote $d(c;[a,b])$ the distance between $c$ and the interval $[a,b]$
(or~$[b,a]$, if~$b<a$).  Finally, a weak convergence criterion for processes
with trajectories in $\ED^{(M_1)}_0(\RR)$ is given by the following
Prokhorov-type theorem.

\begin{theorem}
Let\/ $Y_n$, $n\in \NN^*$, and\/ $Y$ be stochastic process with trajectories
in\/ $\ED_0(\RR)$.  Then,\/ $Y_n$ converges weakly to\/ $Y$ in\/
$\ED_0^{(M_1)}(\RR)$ if and only if:
\begin{itemize}
\item the finite-dimensional distributions of\/ $Y_n$ converge to those of\/
  $Y$ on a dense subset\/~$T$ of\/~$\RR$;
\item for all\/ $\varepsilon>0$, we have
\[
\lim_{h\to 0} \ \mylimsup_{n\to +\infty}
\Pb\bigl(\Delta^{(M_1)}_h(Y_n)>\varepsilon\bigr)=0.
\]
\end{itemize}
\end{theorem}

Note that as the set $T$ in the first condition (here, as well as in
Theorem~\ref{TconvJ1} above) we can take the set
\[
T_Y=\bigl\{u\in\RR \: : \: \Pb\bigl(Y(u-)\ne Y(u)\bigr)=0\bigr\},
\]
which is not only a dense subset of $\RR$, but has an (at most) countable
complement.

Let us also recall that the ``supremum over a (finite or infinite) interval''
functionals are continuous in all the considered function spaces (as well as
in the spaces $\ED^{(J_2)}_0(\RR)$ and~$\ED^{(M_2)}_0(\RR)$ which we do not
consider here), under the condition that the limiting function is continuous
at the endpoints of the interval.  For example, a particular continuous
mapping result that we will need in case of the space $\ED^{(M_1)}_0(\RR)$ is
the following (here and in the sequel, ``$\Longrightarrow$'' denotes the
convergence in distribution).

\begin{proposition}
\label{ContSup}
Let\/ $Y_n$, $n\in \NN^*$, and\/ $Y$ be stochastic process with trajectories
in\/ $\ED_0(\RR)$ such that\/ $Y_n$ converges weakly to\/ $Y$ in\/
$\ED_0^{(M_1)}(\RR)$, and let\/ $x\in T_Y$.  Then
\[
\sup_{u\leq x} Y_n(u) - \sup_{u>x} Y_n(u) \Longrightarrow \sup_{u\leq x} Y(u)
-\sup_{u>x} Y(u).
\]
\end{proposition}

Note finally that the likelihood ratio method often allows to obtain the
properties of the estimators uniformly with respect to $\theta \in \KK$, where
$\KK$ is an arbitrary compact subset of~$\theta$.  For this, the weak
convergence of the likelihood ratios $Z_n=Z_n^{(\theta)}$ to the limiting
likelihood ratio $Z=Z^{(\theta)}$ must be established uniformly with respect
to~$\theta\in \KK$, and so we need some ``uniform versions'' of the weak
convergence criteria given above.  For the case of the
space~$\ED^{(M_1)}_0(\RR)$, one such version can be stated as follows.

\begin{theorem}
\label{conv-faible-M1}
Let\/ $\Sigma$ be a set, and let\/ $Y_n^{(\varsigma)}$ and\/
$Y^{(\varsigma)}$, $n\in \NN^*$ and\/ $\varsigma\in\Sigma$, be stochastic
process with trajectories in\/ $\ED_0(\RR)$.  We suppose that:
\begin{itemize}
\item the finite-dimensional distributions of\/ $Y_n^{(\varsigma)}$ converge
  to those of\/ $Y^{(\varsigma)}$ uniformly with respect to\/
  $\varsigma\in\Sigma$;
\item for all\/ $\varepsilon>0$, we have
\[
\lim_{h\to 0} \ \mylimsup_{n\to +\infty}\ \sup_{\varsigma\in\KK}
\Pb\Bigl(\Delta^{(M_1)}_h\bigl(Y_n^{(\varsigma)}\bigr)>\varepsilon\Bigr)=0.
\]
\end{itemize}
Then,\/ $Y_n^{(\varsigma)}$ converges weakly to\/ $Y^{(\varsigma)}$ in\/
$\ED^{(M_1)}_0(\RR)$ uniformly with respect to\/ $\varsigma\in\Sigma$.
\end{theorem}

In the remainder of this section, we show how the techniques used in the
likelihood ratio analysis method in the case of the space $\ED_0^{(J_1)}(\RR)$
can be adapted to the case of the space $\ED_0^{(M_1)}(\RR)$.

\subsection{Establishing the weak convergence of the likelihood ratio in
  the space $\ED_0^{(M_1)}(\RR)$}

In order to show that the likelihood ratio process $Z_n$ converges weakly to a
limiting likelihood ratio process~$Z=Z^{(\theta)}$ in one of the function
spaces described previously, one needs, in particular, to check the second
condition (the one on the modulus of continuity) of the corresponding weak
convergence criterion.  In the framework of the space $\ED_0^{(J_1)}(\RR)$,
this condition is checked by controlling at the same time the restricted to
intervals of length~$1$ modulus of continuity, as well as the tails of $Z_n$
(cf.~Section~5.3 of \citet{IbHas81}).  The following theorem shows that this
technique is also valid in the case of the space~$\ED_0^{(M_1)}(\RR)$.  Before
we state it, let us slightly rewrite the modulus of continuity~\eqref{MC0} as
\begin{equation}
\label{MC}
\Delta^{(M_1)}_h(f)=\sup_{u\in \RR } \ \ \sup_{u',u''\in \RR \: : \: u-h\leq
  u'\leq u\leq u''\leq u+h} d\bigl( f(u);[f(u'), f(u'')] \bigr) +
\sup_{|u|>1/h} \bil|f(u)\bir|
\end{equation}
and introduce the restricted to an interval $[A,B]\subset\RR$ modulus of
continuity
\begin{equation}
\label{RMC}
\Delta^{(M_1)}_h(f;[A, B])=\sup_{u\in [A,B]} \ \ \sup_{u',u''\in \RR \: : \:
  u-h\leq u'\leq u\leq u''\leq u+h} d\bigl( f(u);[f(u'), f(u'')] \bigr),
\end{equation}
where $f\in \ED_0(\RR)$ and $h>0$.

\begin{theorem}
\label{Cv-M1-unif}
Let\/ $\KK\subset\Theta$ be a compact.  We suppose that:
\begin{enumerate}[label=$(\mathcal{\Alph*})$, ref=$(\mathcal{\Alph*})$]
\item\label{CondA} there exists\/ $\gamma>0$ such that for all\/ $h>0$, $n\in
  \NN^*$, $l\in \ZZ$ and\/ $\theta \in \KK$, we have
\begin{equation}
\label{C1}
\Pb_\theta\Bigl( \Delta^{(M_1)}_h\bigl(Z_n^{1/2};[l, l+1]\bigr)>h^\gamma
\Bigr)\leq B(l)\, h^\gamma,
\end{equation}
where\/ $B$ is a function of (at most) polynomial growth;
\item\label{CondB} there exists\/ $\kappa>0$ such that
\begin{equation}
\label{C2}
\Ex_\theta Z_n^{1/2}(u) \leq \exp\{-\kappa \, \bil| u \bir|\}
\end{equation}
for all\/ $n\in \NN^*$, $u\in \RR$ and\/ $\theta \in \KK$.
\end{enumerate}
Then:
\begin{enumerate}[label=$(\roman*)$,ref=$(\roman*)$]
\item\label{Res1} there exist\/ $b, C>0$ such that for all\/ $n\in \NN^*$,
  $l\in \ZZ$ and\/ $\theta \in \KK$, we have
\begin{equation}
\label{R1}
\Pb_\theta\biggl( \sup_{l\leq u \leq l+1} Z_n(u)>\exp\{-b\, \bil| l \bir|\}
\biggr)\leq C \exp\{-b\,|l|\}\; ;
\end{equation}
\item\label{Res2} there exist\/ $b, C>0$ such that for all\/ $n\in \NN^*$,
  $A\in \NN$ and\/ $\theta \in \KK$, we have
\begin{equation}
\label{R2}
\Pb_\theta\biggl( \sup_{|u|>A} Z_n(u)>\exp\{-b\,A\} \biggr)\leq C
\exp\{-b\,A\}\; ;
\end{equation}
\item\label{Res3} for all\/ $\varepsilon>0$, we have
\[
\lim_{h\to 0} \ \mylimsup_{n\to +\infty} \ \sup_{\theta \in \KK} \Pb_\theta
\bigl(\Delta^{(M_1)}_h(Z_n)>\varepsilon \bigr)=0.
\]
\end{enumerate}
\end{theorem}

Before proving this theorem, let us note that in presence of the convergence
(uniform with respect to $\theta\in\KK$) of the finite-dimensional
distributions of the normalized likelihood ratio process $Z_n$ to those of
some process $Z$, the assertion~\ref{Res3} allows to apply
Theorem~\ref{conv-faible-M1}, and thus to obtain the weak convergence (uniform
with respect to~$\theta\in\KK$) of $Z_n$ to~$Z$ in~$\ED_0^{(M_1)}(\RR)$.
We'll see later that the rate of convergence and the limiting distribution of
the MLE can then be deduced from this weak convergence (as to BEs, their
properties does not rely on a weak convergence in a function space: the
convergence of finite-dimensional distributions, together with some
inequalities on the increments and on the tails of the likelihood ratio
processes, is sufficient).  As to the assertion~\ref{Res2}, it equally plays
an important role in the study of the MLE.  In particular, together with the
rate of convergence and the limiting distribution of the estimator, it allows
to establish the convergence of (polynomial) moments in exactly the same way
as in Section~5.4 of \citet{IbHas81}.

\begin{proof}
Note first that~\ref{Res1} is equivalent to an apparently weaker assertion
that there exist some constants $b_1,b_2,C>0$ such that
\begin{equation}
\label{ineg3}
\Pb_\theta\biggl( \sup_{l\leq u \leq l+1} Z_n(u)>\exp\{-b_1\, \bil| l \bir|\}
\biggr)\leq C \exp\{-b_2\,|l|\}.
\end{equation}
Indeed, taking $b\leq \min\{b_1,b_2\}$, we have $\exp\{-b_i\, \bil| l
\bir|\}\leq \exp\{-b\, \bil| l \bir|\}$, $i=1,2$, and thus
\begin{align*}
\Pb_\theta\biggl( \sup_{l\leq u \leq l+1} Z_n(u)>\exp\{-b\, \bil| l
\bir|\}\biggr) &\leq \Pb_\theta\biggl( \sup_{l\leq u \leq l+1}
Z_n(u)>\exp\{-b_1\, \bil| l \bir|\}\biggr)\\*
&\leq C\exp\{-b_2\, \bil| l \bir|\}\\*
&\leq C\exp\{-b\, \bil| l \bir|\}.
\end{align*}

Now, in order to show~\eqref{ineg3}, we divide the interval $[l, l+1]$ into
$m$ equal parts (we will fix the value of $m$ later) using the points
$l=u_0<u_1<\cdots <u_m=l+1$.  Therefore, for all $l\leq u\leq l+1$, there
exists $1\leq k\leq m$ such that $u_{k-1}\leq u\leq u_{k}$, and we have
\begin{align*}
Z_n^{1/2}(u) &\leq \max\bigl\{Z_n^{1/2}(u_{k-1}),Z_n^{1/2}(u_{k})\bigr\} +
d\bigl( Z_n^{1/2}(u); [Z_n^{1/2}(u_{k-1}), Z_n^{1/2}(u_{k})] \bigr)\\*
&\leq \max\bigl\{Z_n^{1/2}(u_{k-1}),Z_n^{1/2}(u_{k})\bigr\} +
\Delta^{(M_1)}_{1/m}\bigl( Z_n^{1/2}; [l, l+1] \bigr)
\end{align*}
and, consequently,
\begin{equation}
\label{ineg4}
\sup_{l\leq u \leq l+1}Z_n^{1/2}(u) \leq \max_{0\leq i \leq m}Z_n^{1/2}(u_i) +
\Delta^{(M_1)}_{1/m}\bigl( Z_n^{1/2}; [l, l+1] \bigr).
\end{equation}
Thus, for all $b>0$, we can write
\begin{align*}
\Pb_\theta\biggl( \sup_{l\leq u \leq l+1} Z_n^{1/2}(u)>\exp\{-b\,\bil|l\bir|\}
\biggr) &\leq \Pb_\theta\biggl( \max_{0\leq i \leq m} Z_n^{1/2}(u_i)
>\frac{\exp\{-b\,\bil|l\bir|\}}{2} \biggr)\\*
&\phantom{{}\leq}+ \Pb_\theta\biggl( \Delta^{(M_1)}_{1/m}\bigl( Z_n^{1/2}; [l,
  l+1] \bigr)>\frac{\exp\{-b\,\bil|l\bir|\}}{2} \biggr).
\end{align*}

For the first term, for all $b'>b$, we have
\begin{align*}
\Pb_\theta\biggl( \max_{0\leq i \leq m}
Z_n^{1/2}(u_i)>\frac{\exp\{-b\,\bil|l\bir|\}}{2}\biggr) &\leq
\Pb_\theta\biggl( \max_{0\leq i \leq m}
Z_n^{1/2}(u_i)>\frac{\exp\{-b'\,\bil|l\bir|\}}{2}\biggr)\\*
&\leq \sum_{i=0}^m \Pb_\theta\biggl(
Z_n^{1/2}(u_i)>\frac{\exp\{-b'\,\bil|l\bir|\}}{2}\biggr).
\end{align*}
Further, using Markov's inequality and~\eqref{C2}, for $0\leq l\leq u \leq
l+1$, we obtain
\begin{equation}
\label{bloc1}
\begin{aligned}
\Pb_\theta\biggl(Z_n^{1/2}(u)>\frac{\exp\{-b'\,\bil|l\bir|\}}{2} \biggr) &\leq
2\exp\{b'\,\bil|l\bir|\} \, \Ex_\theta Z_n^{1/2}(u)\\
&\leq 2\exp\{b'\,\bil|l|\} \exp\{-\kappa \, \bil|u\bir| \}\\
&\leq 2\exp\{b'\,\bil|l\bir|\} \exp\{-\kappa \, \bil|l\bir| \}\\
&= 2\exp\{(b'-\kappa)\bil|l\bir|\},
\end{aligned}
\end{equation}
and for $l\leq u \leq l+1\leq 0$, we obtain
\begin{align*}
\Pb_\theta\biggl(Z_n^{1/2}(u)>\frac{\exp\{-b'\,\bil|l\bir|\}}{2} \biggr) &\leq
2\exp\{b'\,\bil|l|\} \exp\{-\kappa \, \bil|u\bir| \}\\*
&\leq 2\exp\{b'\,\bil|l\bir|\} \exp\{-\kappa \, (\bil|l\bir|-1) \}\\*
&= 2e^{\kappa} \exp\{(b'-\kappa)\bil|l\bir|\}.
\end{align*}
Therefore,
\[
\Pb_\theta\biggl( \max_{0\leq i \leq m}
Z_n^{1/2}(u_i)>\frac{\exp\{-b\,\bil|l\bir|\}}{2}\biggr) \leq \sum_{i=0}^m
2e^{\kappa} \exp\{(b'-\kappa)\bil|l\bir|\} = 2e^{\kappa}(m+1)
\exp\{(b'-\kappa)\bil|l\bir|\}.
\]

In order to upper bound the second term, we choose $m=\lfloor
2^{1/\gamma}\exp\{b\, \bil| l \bir|/\gamma\}\rfloor +1$.  Since~$m\geq
2^{1/\gamma}\exp\{b\, \bil| l \bir|/\gamma\}$, we have $(1/m)^\gamma \leq
\exp\{-b\, \bil| l \bir|\}/2$, and hence
\[
\Pb_\theta\biggl( \Delta^{(M_1)}_{1/m}\bigl( Z_n^{1/2}; [l, l+1]
\bigr)>\frac{\exp\{-b\,\bil|l\bir|\}}{2} \biggr) \leq \Pb_\theta\biggl(
\Delta^{(M_1)}_{1/m}\bigl( Z_n^{1/2}; [l, l+1] \bigr)>
\frac{1}{m^\gamma}\biggr) \leq \frac{B(l)}{m^\gamma}\, .
\]

So, by choosing $b$ and $b'$ such that $b<b'$ and $(1/2+1/\gamma)b+b'<
\kappa$, we obtain
\begin{align*}
\Pb_\theta\biggl( \sup_{l\leq u \leq l+1} {}& Z_n(u)>\exp\{-b\,|l|\} \biggr)
\leq \Pb_\theta\biggl( \sup_{l\leq u \leq l+1} Z_n^{1/2}(u)>\exp\{-b\,|l|/2\}
\biggr)\\*
&\leq \Pb_\theta\biggl( \sup_{l\leq u \leq l+1} Z_n^{1/2}(u)>\exp\{-b\,|l|\}
\biggr)\\
&\leq 2e^{\kappa} (m+1)\exp\{(b'-\kappa)\bil|l\bir|\}+\frac{B(l)}{m^\gamma}\\
&\leq 2e^{\kappa}\bigl( 2^{1/\gamma}\exp\{b\, \bil|l\bir|/\gamma\}+2\bigr)
\exp\{(b'-\kappa)\bil|l\bir|\} + \frac{B(l)}{2}\exp\{-b\,\bil|l\bir|)\\
&= 2^{1/\gamma+1}e^{\kappa}\exp\{(b/\gamma+b'-\kappa\}\bil|l\bir|\}
+4e^{\kappa} \exp\{(b'-\kappa)\bil|l\bir|\} +
\frac{B(l)}{2}\exp\{-b\,\bil|l\bir|\}\\
&= \exp\{-b\,|l|/2\}\biggl[ 2^{1/\gamma+1}e^{\kappa} \exp \bigl\{
  \bigl((1/2+1/\gamma)b+b'-\kappa \bigr)\bil|l\bir|\bigr\}\\*
&\phantom{{}=\exp\{-b\,|l|/2\}\biggl[}+
    4e^{\kappa}\exp\{(b/2+b'-\kappa)\bil|l\bir|\} +
    \frac{B(l)}{2}\exp\{-b\,\bil|l\bir|/2\} \biggr]\\
&\leq \exp\{-b\,|l|/2\} \biggl[ 2^{1/\gamma+1}e^{\kappa}+4e^{\kappa} +
    \frac{B(l)}{2}\exp\{-b\,\bil|l\bir|/2\} \biggr]\\*
&\leq C\exp\{-b\,|l|/2\},
\end{align*}
where we used the fact that $B(l) \exp\{-b\,\bil|l\bir|/2\}$ in bounded.  As
noted above, this concludes the proof of~\ref{Res1}.

Now, let us turn to the proof of~\ref{Res2}.  We have
\begin{align*}
\Pb_\theta\biggl( \sup_{|u|>A} Z_n(u)>\exp\{-b\, A\} \biggr) &\leq \sum_{l\geq
  0} \Pb_\theta\biggl( \sup_{A+l\leq u\leq A+l+1} Z_n(u)>\exp\{-b\,A\}
\biggr)\\*
&\phantom{{}\leq}+ \sum_{l\geq 0} \Pb_\theta\biggl( \sup_{-A-l-1\leq u\leq
  -A-l} Z_n(u)>\exp\{-b\,A\} \biggr).
\end{align*}
For the first sum, we get
\begin{align*}
\sum_{l\geq 0} \Pb_\theta\biggl( \sup_{A+l\leq u\leq A+l+1} {}&
Z_n(u)>\exp\{-b\,A\} \biggr)\\*
&\leq \sum_{l\geq 0} \Pb_\theta\biggl( \sup_{A+l\leq u\leq A+l+1}
Z_n(u)>\exp\{-b(A+l)\} \biggr)\\
&\leq C\sum_{l\geq 0} \exp\{-b(A+l)\}\\
&\leq C\exp\{-b\,A\}\sum_{l\geq 0} \bigl(e^{-b}\bigr)^l\\
&\leq C\exp\{-b\,A\}\, \frac{1}{1-e^{-b}}\\*
&= C'\exp\{-b\,A\}.
\end{align*}
The second sum is treated in a similar way, which concludes the proof
of~\ref{Res2}.

It remains to show~\ref{Res3}.  We start by checking that the following
inequality is valid as soon as $h<1/A$:
\begin{equation}
\label{ineg1}
\Pb_\theta \Bigl(\Delta^{(M_1)}_h(Z_n)>\varepsilon \Bigr) \leq \Pb_\theta
\Bigl(\Delta^{(M_1)}_h(Z_n;[-A, A])>\frac{\varepsilon}{2} \Bigr) + \Pb_\theta
\biggl(\sup_{|u|>A}Z_n(u)>\frac{\varepsilon}{2} \biggr).
\end{equation}
For this, it is sufficient to show the inclusion
\begin{equation}
\label{inclusion1}
\Bigl\{ \Delta^{(M_1)}_h(Z_n) >\varepsilon \Bigr\} \subset \Bigl\{
\Delta^{(M_1)}_h(Z_n;[-A,A]) >\frac{\varepsilon}{2} \Bigr\} \cup \biggl\{
\sup_{|u|>A}(Z_n) >\frac{\varepsilon}{2} \biggr\}.
\end{equation}
Passing to the complement, let us suppose that
\[
\Delta^{(M_1)}_h(Z_n;[-A,A]) \leq \frac{\varepsilon}{2} \qquad \text{and} \qquad
\sup_{|u|>A}(Z_n) \leq \frac{\varepsilon}{2}\, ,
\]
and show that
\[
\Delta^{(M_1)}_h(Z_n) \leq \varepsilon.
\]
For the second term of the definition~\eqref{MC} of $\Delta^{(M_1)}_h(Z_n)$,
since $1/h>A$, we have
\[
\sup_{|u|>1/h}(Z_n) \leq \sup_{|u|>A}(Z_n) \leq \frac{\varepsilon}{2}\, .
\]
In order to upper bound the first term of $\Delta^{(M_1)}_h(Z_n)$, which we
denote
\[
\smash{\overset{\circ}\Delta}^{(M_1)}_h(Z_n)=\sup_{u\in \RR}
\ \ \sup_{u',u''\in \RR \: : \: u-h\leq u'\leq u\leq u''\leq u+h} d\bigl(
Z_n(u),[Z_n(u'), Z_n(u'')] \bigr),
\]
we first upper bound the distance $d\bigl( Z_n(u),[Z_n(u'), Z_n(u'')] \bigr)$.
\begin{enumerate}
\item If $u\in [-A,A]$, we have
\[
d\bigl( Z_n(u),[Z_n(u'), Z_n(u'')] \bigr) \leq \Delta^{(M_1)}_h(Z_n;[-A,A])
\leq \frac{\varepsilon}{2}\, .
\]
\item If $u\notin [-A,A]$, we have either $u'\notin [-A,A]$, or $u''\notin
  [-A,A]$ (assume that it is~$u'$), and so
\[
Z_n(u), Z_n(u') \leq \sup_{|u|>A } Z_n(u) \leq \frac{\varepsilon}{2} \, ,
\]
which yields again
\[
d\bigl( Z_n(u),[Z_n(u'), Z_n(u'')] \bigr) \leq |Z_n(u)-Z_n(u')| \leq
\frac{\varepsilon}{2}\, .
\]
\end{enumerate}
Thus, we obtain
\[
\smash{\overset{\circ}\Delta}^{(M_1)}_h(Z_n) \leq \frac{\varepsilon}{2}
\]
and, consequently, the inclusion~\eqref{inclusion1} and the
inequality~\eqref{ineg1}.

For the second term of~\eqref{ineg1}, for $A$ large enough (such that
$\varepsilon/2>\exp\{-b\,A\}$), we have
\[
\Pb_\theta \biggl(\sup_{|u|>A}Z_n(u)>\frac{\varepsilon}{2} \biggr) \leq
\Pb_\theta \biggl(\sup_{|u|>A}Z_n(u)>\exp\{-b\,A\} \biggr) \leq C
\exp\{-b\,A\}.
\]
So, we can make this term as small as we want by choosing $A$ large enough.

Now, in order to upper bound the first term of~\eqref{ineg1}, we first show
that
\begin{equation}
\label{ineg2}
\Delta^{(M_1)}_h(Z_n;[-A, A]) \leq 2 M_n \Delta^{(M_1)}_h(Z_n^{1/2};[-A, A]),
\end{equation}
where $M_n=\sup_{|u|\leq A+1} Z_n^{1/2}(u)$.  Here and in the sequel we
consider that $h<1$.  Using the elementary inequality
\[
\bil| a^2-b^2 \bir| = \bil| a-b \bir|\, (a+b) \leq 2\,\bil| a-b \bir| \max
\{a,b\} , \qquad a, b\geq 0,
\]
for all $-A\leq u\leq A$ and $u-h\leq v\leq u+h$ we obtain
\begin{align*}
\bil|Z_n(u)-Z_n(v)\bir| &= \bigl|( Z_n^{1/2}(u))^2-( Z_n^{1/2}(v))^2\bigr|\\*
&\leq 2\, \bigl|Z_n^{1/2}(u)-Z_n^{1/2}(v)\bigr| \max\bigl\{
Z_n^{1/2}(u),Z_n^{1/2}(v)\bigr\}\\*
&\leq 2\, \bigl|Z_n^{1/2}(u)-Z_n^{1/2}(v)\bigr|\, M_n,
\end{align*}
which yields~\eqref{ineg2}.

So, for all $0<h<1$, we have
\begin{align*}
\Pb_\theta \Bigl(\Delta^{(M_1)}_h(Z_n;[-A&, A]) > h^{\gamma/2} \Bigr) \leq
\Pb_\theta \Bigl( 2M_n\Delta^{(M_1)}_h(Z_n^{1/2};[-A,
  A])>h^{\gamma/2}\Bigr)\\*
&\leq \Pb_\theta \Bigl(\Delta^{(M_1)}_h(Z_n^{1/2};[-A, A])>h^{\gamma}\Bigr) +
\Pb_\theta \bigl( 2M_n>h^{-\gamma/2}\bigr)\\
&\leq \sum_{l=-A}^{A-1} \Pb_\theta \Bigl(\Delta^{(M_1)}_h(Z_n^{1/2};[l,
  l+1])>h^{\gamma}\Bigr) + \Pb_\theta \bigl( 2M_n>h^{-\gamma/2}\bigr)\\*
&\leq B'(A)\, h^\gamma + \Pb_\theta \bigl( 2M_n>h^{-\gamma/2}\bigr),
\end{align*}
where $B'(A)=\sum_{l=-A}^{A-1} B(l)$.

In order to upper bound the last probability, we divide the interval $[-A-1,
  A+1]$ into~$m=\lfloor h^{-\gamma/4}\rfloor$ equal parts using the points
$u_0=-A-1<u_1<\cdots<u_m=A+1$.  Using an argument similar to the one which was
used to show~\eqref{ineg4}, we get
\[
M_n\leq \max_{0\leq i \leq m} Z_n^{1/2}(u_i) +
\Delta^{(M_1)}_{\frac{2A+2}{m}}(Z_n^{1/2};[-A-1, A+1]),
\]
and consequently,
\begin{align*}
\Pb_\theta \biggl( M_n>\frac{h^{-\gamma/2}}{2}\biggr) &\leq \Pb_\theta \biggl(
\max_{0\leq i \leq m} Z_n^{1/2}(u_i)>\frac{h^{-\gamma/2}}{4}\biggr)\\*
&\phantom{{}\leq} + \Pb_\theta
\biggl(\Delta^{(M_1)}_{\frac{2A+2}{m}}(Z_n^{1/2};[-A-1,
  A+1])>\frac{h^{-\gamma/2}}{4}\biggr).
\end{align*}
For the first term, since
\[
\Ex Z_n^{1/2}(u) \leq \exp\{-\kappa \, \bil| u \bir| \} \leq 1,
\]
$m\leq h^{-\gamma/4}$ and $h<1$, we obtain
\begin{align*}
\Pb_\theta \biggl( \max_{0\leq i \leq m}
Z_n^{1/2}(u_i)>\frac{h^{-\gamma/2}}{4}\biggr) &\leq \sum_{i=0}^m \Pb_\theta
\biggl( Z_n^{1/2}(u_i)>\frac{h^{-\gamma/2}}{4}\biggr)\\*
&\leq \sum_{i=0}^m \frac{ \Ex Z_n^{1/2}(u_i)}{h^{-\gamma/2}/4}\\
&\leq 4 h^{\gamma/2} (m+1)\\*
&\leq 8h^{\gamma/4}.
\end{align*}
To upper bound the second term, let us first note that as
\[
\frac{h^{-\gamma /2}}{4} \xrightarrow[h\to 0]{ } +\infty \qquad \text{and}
\qquad \biggl(\frac{2A+2}{m}\biggr)^\gamma=\biggl(\frac{2A+2}{\lfloor
  h^{-\gamma /4}\rfloor}\biggr)^\gamma \xrightarrow[h\to 0]{ } 0,
\]
for $h$ sufficiently small we have $h^{-\gamma
  /2}/4>\bigl((2A+2)/m\bigr)^\gamma$.  Therefore, using the fact that~$m\geq
h^{-\gamma/4}/2$, we can write
\begin{align*}
\Pb_\theta \biggl(\Delta^{(M_1)}_{\frac{2A+2}{m}}(Z_n^{1/2};[-A-1&,
  A+1])>\frac{h^{-\gamma/2}}{4}\biggr)\\*
&\leq \Pb_\theta \Biggl(\Delta^{(M_1)}_{\frac{2A+2}{m}}(Z_n^{1/2};[-A-1,
  A+1])>\biggl(\frac{2A+2}{m}\biggr)^\gamma\Biggr)\\
&\leq \sum_{l=-A-1}^{A} B(l)\, \biggl(\frac{2A+2}{m}\biggr)^\gamma\\
&= B'(A+1)\, \biggl(\frac{2A+2}{m}\biggr)^\gamma\\*
&\leq (4A+4)^{\gamma}\, B'(A+1)\,h^{\gamma^2/4},
\end{align*}
which yields
\[
\Pb_\theta \bigl( 2M_n>h^{-\gamma/2}\bigr) \leq
8h^{\gamma/4}+(4A+4)^{\gamma}\,B'(A+1)\,h^{\gamma^2/4}.
\]

So, for $h$ sufficiently small, we finally obtain
\begin{align*}
\Pb_\theta \Bigl(\Delta^{(M_1)}_h(Z_n;[-A, A])>\frac{\varepsilon}{2}\Bigr)
&\leq \Pb_\theta \Bigl(\Delta^{(M_1)}_h(Z_n;[-A, A])>h^{\gamma/2}\Bigr)\\*
&\leq B'(A)\, h^\gamma +8h^{\gamma/4}+(4A+4)^{-\gamma}\, B'(A+1)\,
h^{\gamma^2/4},
\end{align*}
and hence this term becomes as small as we want when $h\to 0$, which completes
the proof of~\ref{Res3}.
\end{proof}

We conclude this subsection with some remarks which allow to enlarge the scope
of application of Theorem~\ref{Cv-M1-unif}.
\begin{remark}
\label{CondA'-Cv-M1}
Theorem~\ref{Cv-M1-unif} remains valid if instead of~\ref{CondA} we suppose
\begin{enumerate}[label=$(\mathcal{\Alph*'})$]
\itshape
\item\label{CondA'} there exist\/ $\gamma_1,\gamma_2,h_0>0$ such that for
  all\/ $0<h<h_0$, $n\in\NN^*$, $l\in \ZZ$ and\/ $\theta \in \KK$, we have
\begin{equation}
\label{C1b}
\Pb_\theta\Bigl( \Delta^{(M_1)}_h\bigl(Z_n^{1/2};[l,l+1]\bigr)>h^{\gamma_1}
\Bigr)\leq B(l)\, h^{\gamma_2},
\end{equation}
where\/ $B$ is a function of (at most) polynomial growth.
\end{enumerate}
Indeed, taking $\gamma \leq \min \{ \gamma_1, \gamma_2 \}$, for $0<h< h_0$ we
have
\[
\Pb_\theta\Bigl( \Delta^{(M_1)}_h\bigl(Z_n^{1/2};[l, l+1]\bigr)>h^\gamma
\Bigr) \leq \Pb_\theta\Bigl( \Delta^{(M_1)}_h\bigl(Z_n^{1/2};[l,
  l+1]\bigr)>h^{\gamma_1} \Bigr) \leq B(l)\, h^{\gamma_2} \leq B(l)\,
h^\gamma,
\]
and for $h\geq h_0$ we have
\[
\Pb_\theta\Bigl( \Delta^{(M_1)}_h\bigl(Z_n^{1/2};[l, l+1]\bigr)>h^\gamma
\Bigr) \leq 1 \leq \frac{1}{h_0^{\gamma}} \, h^{\gamma}.
\]
Thus, for all $h>0$, we obtain
\[
\Pb_\theta\Bigl( \Delta^{(M_1)}_h\bigl(Z_n^{1/2};[l, l+1]\bigr)>h^\gamma \Bigr)
\leq
 B^\star(l) \, h^{\gamma}
\]
where $B^\star(l)=\max \{B(l), (1/h_0)^\gamma \}$, and so~\ref{CondA} is
verified.
\end{remark}

\begin{remark}
\label{CondB'-Cv-M1}
Theorem~\ref{Cv-M1-unif} remains valid if instead of~\ref{CondB} we suppose
\begin{enumerate}[label=$(\mathcal{\Alph*'})$, start=2]
\itshape
\item there exist\/ $\kappa>0$ and a non negative and increasing
  function\/~$g$ on\/ $[ 0,+\infty)$ verifying\/ $g(l)\geq \kappa \, l$ for
    all\/ $l \in \NN$, such that
\begin{equation}
\label{C2b}
\Ex_\theta Z_n^{1/2}(u) \leq \exp\{-g(\bil| u \bir|)\}
\end{equation}
for all\/ $n\in\NN^*$, $u\in \RR$ and\/ $\theta \in \KK$.
\end{enumerate}
Indeed, it is enough to make some small adaptations in the proof.  For
example, the sequence of inequalities~\eqref{bloc1} becomes
\begin{align*}
\Pb_\theta\biggl(Z_n^{1/2}(u)>\frac{\exp\{-b'\,\bil|l\bir|\}}{2} \biggr) &\leq
2\exp\{b'\,\bil|l\bir|\} \, \Ex_\theta Z_n^{1/2}(u)\\*
&\leq 2\exp\{b'\,\bil|l|\} \exp\{- g(\bil|u\bir|) \}\\
&\leq 2\exp\{b'\,\bil|l|\} \exp\{-g(\bil|l\bir|) \}\\
&\leq 2\exp\{b'\,\bil|l\bir|\} \exp\{-\kappa \, \bil|l\bir| \}\\*
&= 2\exp\{(b'-\kappa)\bil|l\bir|\}.
\end{align*}
\end{remark}

\begin{remark}
\label{n-C1-Cv-M1}
If in the hypotheses of Theorem~\ref{Cv-M1-unif} we suppose that the
inequality~\eqref{C1} is verified only starting from a certain $n_0$ (instead
of being verified for all $n\in \NN^*$), the inequalities~\eqref{R1}
and~\eqref{R2} will obviously hold starting from this same $n_0$, and the
assertion~\ref{Res3} will remain true.  The same considerations apply, of
course, to the inequalities~\eqref{C2}, \eqref{C1b} and~\eqref{C2b}.
\end{remark}

\begin{remark}
We can replace $A\in \NN$ by $A\in [ 1,+\infty )$ in the assertion~\ref{Res2}
  of Theorem~\ref{Cv-M1-unif}.  Indeed, for all $A\geq 1$, we have $A/2 \leq
  \lfloor A \rfloor \leq A$, and therefore we can write
\begin{align*}
\Pb_\theta\biggl( \sup_{|u|>A} Z_n(u)>\exp(-b\, A/2) \biggr) &\leq
\Pb_\theta\biggl( \sup_{|u|>\lfloor A \rfloor} Z_n(u)>\exp(-b\lfloor A
\rfloor) \biggr)\\*
&\leq C\exp(-b\lfloor A \rfloor)\\*
&\leq C\exp(-b \, A/2).
\end{align*}
Thus, we obtain the inequality~\eqref{R2} (with $b/2$ instead of $b$).
\end{remark}

\subsection{Deducing the asymptotic properties of the MLE}

Let us now show how the rate of convergence and the limiting distribution of
the MLE can be deduced from the convergence of $Z_n$ to $Z$ in the
space~$\ED_0^{(M_1)}(\RR)$ under the condition that the limiting process $Z$
almost surely ``attains its suppremum'' in a unique point.  For this, we adapt
(while also treating some technical points more precisely) the argument
developed by \citeauthor{IbHas81} in~\myciteyearp{IbHas81} for the case of
convergence in~$\ED_0^{(J_1)}(\RR)$.

First of all, we recall that in the case where the likelihood $L \bigl(\theta,
X^{(n)} \bigr)$ is a càdlàg function of $\theta$, the MLE $\hat{\theta}_n =
\argsup_{t \in \Theta} L \bigl(t, X^{(n)} \bigr)$ is defined as any of the
solutions of the equation
\[
L^\pm \bigl(\hat{\theta}_n, X^{(n)} \bigr) = \sup_{t \in \Theta} L \bigl(t,
X^{(n)} \bigr),
\]
where we use the notation $L^\pm \bigl(t, X^{(n)} \bigr) = \max\bigl\{ L
\bigl(t-, X^{(n)} \bigr), L \bigl(t, X^{(n)} \bigr) \bigr\}$, $t\in\RR$.  We
equally introduce the smallest and the largest ``possible values'' of the EMV
by
\begin{align*}
\hat{\theta}_n^- &= \inf \Bigl\{s\in\Theta \: : \: L^\pm \bigl(s, X^{(n)}
\bigr) = \sup_{t \in \Theta} L \bigl(t, X^{(n)} \bigr) \Bigr\}\\
\intertext{and}
\hat{\theta}_n^+ &= \sup \Bigl\{s\in\Theta \: : \: L^\pm \bigl(s, X^{(n)}
\bigr) = \sup_{t \in \Theta} L \bigl(t, X^{(n)} \bigr) \Bigr\}
\end{align*}
respectively.  Finally, we suppose that the $\argsup$ of the limiting process
$Z$ is almost surely unique, that is, that the equation
\[
Z^\pm(\eta) = \sup_{u \in \RR} Z(u)
\]
has a unique solution $\eta = \argsup_{u \in \RR} Z(u)$, and show that
$\varphi_n^{-1} \bigl(\hat{\theta}_n-\theta\bigr) \Longrightarrow \eta$ (the
notations~$Z^\pm$, used here, and $Z_n^\pm$, which will be used below, are
analogous to the notation $L^\pm$ introduced above).

We start by noting that it is sufficient to show the convergences
$\varphi_n^{-1} \bigl(\hat{\theta}_n^--\theta\bigr) \Longrightarrow \eta$
and~$\varphi_n^{-1} \bigl(\hat{\theta}_n^+-\theta\bigr) \Longrightarrow \eta$.
We will study only the second one (the first one can be treated in a similar
way).

Using the change of variable $t=\theta+u\varphi_n$, $s=\theta+v\varphi_n$ (and
thus $v=\varphi_n^{-1}(s-\theta)$), we can write
\begin{align*}
\varphi_n^{-1} (\hat{\theta}_n^+-\theta) &= \sup \Bigl\{v \in \UU_n \: : \:
L^\pm \bigl(\theta+v\varphi_n, X^{(n)}\bigr) = \sup_{u \in \UU_n} L
\bigl(\theta+u\varphi_n, X^{(n)}\bigr) \Bigr\}\\*
&= \sup \Bigl\{v \in \UU_n \: : \: Z_n^\pm(v) = \sup_{u \in \UU_n} Z_n(u)
\Bigr\}.
\end{align*}
Therefore, for all $x \in\RR$, denoting $A_{n,x}=\bigl\{Z_n(x-) =
Z_n(x)\bigr\}$ we have the following equality of the events.
\[
\{\varphi_n^{-1} (\hat{\theta}_n^+-\theta)<x\} \cap A_{n,x} = \biggl\{ \sup_{u
  \leq x}Z_n(u) > \sup_{u > x}Z_n(u) \biggr\} \cap A_{N,x},
\]
and hence
\[
\Pb_\theta \bigl( \varphi_n^{-1} (\hat{\theta}_n^+-\theta)<x\bigr) =
\Pb_\theta \biggl( \sup_{u \leq x}Z_n(u) > \sup_{u > x}Z_n(u) \biggr)
\]
as soon as $\Pb(A_{N,x})=1$, that is, for all $x\in T_{Z_n}$.

Similarly, we get
\[
\Pb (\eta<x) = \Pb \biggl( \sup_{u \leq x}Z(u) > \sup_{u > x}Z(u) \biggr)
\]
for all $x\in T_Z$.

Now, according to Proposition~\ref{ContSup}, we have the convergence in
distribution
\[
\sup_{u \leq x}Z_n(u) - \sup_{u > x}Z_n(u) \Longrightarrow \sup_{u \leq
  x}Z(u) - \sup_{u > x}Z(u),
\]
and hence
\begin{equation}
\label{ConvSupGtSup}
\Pb_\theta \biggl( \sup_{u \leq x}Z_n(u) > \sup_{u > x}Z_n(u) \biggr)
\longrightarrow \Pb \biggl( \sup_{u \leq x}Z(u) > \sup_{u > x}Z(u) \biggr),
\end{equation}
provided that
\[
\Pb \biggl( \sup_{u \leq x}Z(u) = \sup_{u > x}Z(u) \biggr)=0.
\]
However, taking into account that the $\argsup$ of $Z$ is almost surely
unique, we have
\[
\Pb \biggl( \sup_{u \leq x}Z(u) = \sup_{u > x}Z(u) \biggr) = \Pb
(\eta=x)
\]
for all $x\in T_Z$ (note that, as above, this equality between probabilities
is true not because of the equality of the concerned events, but because of
the equality of their intersections with the event $A_x = \bigl\{Z(x-) =
Z(x)\bigr\}$ of probability~$1$).  So, the convergence~\eqref{ConvSupGtSup}
holds for all~$x \in T = \bigl\{x \: : \: \Pb(\eta=x) = 0\bigr\}$.  Note also
that the complement of the set $T$ is the set of atoms of the random variable
$\eta$ and is, consequently, (at most) countable.

Thus, we have shown that
\begin{equation}
\label{ConvCDFMLE}
\Pb_\theta \bigl( \varphi_n^{-1} (\hat{\theta}_n^+-\theta)<x\bigr)
\longrightarrow
\Pb (\eta<x)
\end{equation}
for all
\[
x \in S = \biggl(\bigcap_{n\in\NN^*} T_{Z_n}\biggr) \cap T_Z \cap T.
\]

It remains to note that the complement of the set $S$ being a countable union
of (at most) countable sets, it is (at most) countable itself.  In particular,
the convergence~\eqref{ConvCDFMLE} holds for all $x$ from a dense subset of
$\RR$ (the set $S$), which is sufficient to conclude the desired convergence
in distribution $\varphi_n^{-1} \bigl(\hat{\theta}_n^+-\theta\bigr)
\Longrightarrow \eta$.  As we have already said, together with $\varphi_n^{-1}
\bigl(\hat{\theta}_n^- -\theta\bigr) \Longrightarrow \eta$, this allows us to
deduce the rate of convergence and the limiting distribution of
$\hat{\theta}_n$:
\[
\varphi_n^{-1} \bigl(\hat{\theta}_n-\theta\bigr) \Longrightarrow \eta.
\]
Note also that afterwards, as already mentioned in the previous subsection,
the convergence of (polynomial) moments, that is,
\[
\lim_{n\to +\infty} \varphi_n^{-p} \, \Ex_\theta
\bigl|\hat{\theta}_n-\theta\bigr|^p = \Ex {\bil| \eta \bir|}^p
\]
for any $p>0$, can be established in exactly the same way as in Section~5.4 of
\citet{IbHas81} using the assertion~\ref{Res2} of Theorem~\ref{Cv-M1-unif}.

To conclude this section, let us note that the assumption that the $\argsup$
of the limiting process $Z$ is almost surely unique is absolutely crucial.
Without this assumption, it is not even true, in general, that the smallest
(resp.~the largest) point of maximum of $Z_n$ converges in distribution to the
smallest (resp.~the largest) point of maximum of $Z$.  However, we can mention
the paper \citet{SeiSen}, where such a result was established for processes on
a compact interval with trajectories in the space
$\ED^{(J_1)}\bigl([a,b]\bigr)$, although only in a very particular case.
Essentially, they consider only processes with piecewise constant trajectories
(it should be noted, though, that their general framework is multivariate,
with trajectories piecewise constant with respect to one variable, and
continuous with respect to the others), they suppose that the limiting process
almost surely attains its maximum on a unique plateau, and they require not
only the weak convergence of the processes themselves, but also the weak
convergence of the associated \emph{pure jump processes\/} (which have the
same jump points, but whose all the jumps are of size $1$).

\section{Some tools for the study of the modulus of continuity corresponding
  to the Skorokhod $M_1$ topology}
\label{Sec-outils}

In this section, we develop some tools that can facilitate the study of the
$M_1$~modulus of continuity (the modulus of continuity corresponding to the
Skorokhod $M_1$ topology).

To simplify the exposition, we place ourselves in the original framework of
functions on a compact interval $[a,b]\subset\RR$ considered by
\citeauthor{Sko} in~\myciteyearp{Sko}.  Thus, we study the modulus of
continuity corresponding to the topology~$M_1$ on the space
$\ED\bigl([a,b]\bigr)$ of càdlàg functions on~$[a,b]$.  This modulus of
continuity is given by
\[
\omega_h^{(M_1)}(f)=\sup_{u,u',u''\in[a,b] \: : \: u-h\leq u'\leq u\leq
  u''\leq u+h} d\bigl( f(u);[f(u'),f(u'')] \bigr),
\]
where $f \in \ED\bigl([a,b]\bigr)$ and $h>0$.  We equally recall the space
$\EC\bigl([a,b]\bigr)$ of continuous functions on~$[a,b]$, as well as the
uniform modulus of continuity (corresponding to the uniform topology~$U$ on
either of these two spaces)
\[
\omega_h^{(U)}(f) = \sup_{u,u'\in[a,b]\::\:\bil|u-u'\bir|\leq h} |f(u)-f(u')|.
\]

Note that though all the notions that we introduce below, as well as most of
the results that we obtain, can be directly transposed to the case of
functions on the whole real line, this is unfortunately not the case of the
results concerning the modulus of continuity~$\omega_h^{(M_1)}(f)$ (neither of
those concerning the modulus of continuity~$\omega_h^{(U)}(f)$), because of
the second term of the definition~\eqref{MC} of the modulus of
continuity~$\Delta^{(M_1)}_h(f)$ of functions~$f\in\ED_0(\RR)$.  Nevertheless,
these results can be adapted (we will give more details at the end of this
section) to the restricted to an interval $[A,B]\subset\RR$ modulus of
continuity of functions $f\in\ED_0(\RR)$, that is, to
$\Delta^{(M_1)}_h(f;[A,B])$ defined by~\eqref{RMC}.  They can hence be useful
for checking the condition~\ref{CondA} of Theorem~\ref{Cv-M1-unif} which, as
we have seen in the previous section, is the result allowing to control the
modulus of continuity~$\Delta^{(M_1)}_h(f)$ of functions~$f\in\ED_0(\RR)$ in
the framework of the likelihood ratio analysis method.

\subsection{Moduli of increase and decrease}

We introduce the moduli of increase and decrease with the help of the
following definition.
\begin{definition}
Let\/ $f$ be a function on an interval\/ $[a,b]\subset\RR$, and let\/ $h>0$.
\begin{enumerate}
\item We call\/ \emph{modulus of increase} of\/ $f$ the quantity
\[
\omega^+_h(f) = \sup_{u,v\in[a,b]\::\:u\leq v\leq u+h} \bigl( f(v)-f(u)
\bigr)_+ \: .
\]
\item We call\/ \emph{modulus of decrease} of\/ $f$ the quantity
\[
\omega^-_h(f) = \omega^+_h(-f) = \sup_{u,v\in[a,b]\::\:u\leq v\leq u+h} \bigl(
f(v)-f(u) \bigr)_- \: .
\]
\end{enumerate}
\end{definition}

Recall that for $x\in\RR$, as usual, $x_+$ and $x_-$ denote respectively its
positive and negative parts which are given by
\[
x_+=\max\{x,0\}=\frac{\bil|x\bir|+x}{2} \qquad \text{and} \qquad
x_-=(-x)_+=-\min\{x,0\}=\frac{\bil|x\bir|-x}{2} \, ,
\]
and that we have $x = x_+ - x_-$ and $\bil|x\bir| = x_+ +x_- =
\max\{x_+,x_-\}$.

Let us note that after introducing the notions of the moduli of increase and
decrease, we discovered that these terms have already been used in
\citet*{Appell} and in \citet{Banas} for some very similar objects.  The
difference is that in their definition $\bil|x\bir|+x$ and $\bil|x\bir|-x$ are
used instead of the positive and negative parts respectively (which results in
two times greater quantities).

Note also that if $f$ is an increasing function, we obviously have
$\omega^-_h(f)=0$ for all~$h>0$.  Similarly, if $f$ is a decreasing function,
we have $\omega^+_h(f)=0$ for all~$h>0$.

The following results relate the uniform and $M_1$ moduli of continuity
on~$\ED\bigl([a,b]\bigr)$ to the moduli of increase and decrease.
\begin{proposition}
\label{module-de-de/croissance-avec-U}
Let\/ $f\in \ED\bigl([a,b]\bigr)$ and\/ $h>0$.  Then,
\[
\omega^{(U)}_h(f) = \max\bigl\{ \omega^+_h(f), \omega^-_h(f)\bigr\}.
\]
\end{proposition}

\begin{proof}
We have
\begin{align*}
\omega_h^{(U)}(f) &= \sup_{u,u'\in[a,b]\::\:\bil|u-u'\bir|\leq h}
|f(u)-f(u')|\\*
&= \sup_{u,v\in[a,b]\::\:u\leq v\leq u+h} |f(v)-f(u)|\\
&= \sup_{u,v\in[a,b]\::\:u\leq v\leq u+h} \max\bigl\{ \bigl( f(v)-f(u)
\bigr)_+ , \bigl( f(v)-f(u) \bigr)_-\bigr\}\\*
&= \max\bigl\{ \omega^+_h(f), \omega^-_h(f)\bigr\},
\end{align*}
which concludes the proof.
\end{proof}

\begin{theorem}
\label{module-de-de/croissance-avec-M1}
Let\/ $f\in \ED\bigl([a,b]\bigr)$ and\/ $h>0$.  Then,
\begin{equation}
\label{ineg-moduleM1}
\omega_h^{(M_1)}(f) \leq \min\bigl\{ \omega^+_h(f), \omega^-_h(f)\bigr\}.
\end{equation}
\end{theorem}

\begin{proof}
First of all, let us recall that
\[
\omega_h^{(M_1)}(f)
=
\sup_{u,u',u''\in[a,b] \: : \: u-h\leq u'\leq u\leq u''\leq u+h} d\bigl( f(u);[f(u'), f(u'')] \bigr).
\]

If $f(u) \in [f(u'), f(u'')]$, we have $d\bigl( f(u);[f(u'), f(u'')] \bigr) = 0 \leq \min\bigl\{ \omega^+_h(f), \omega^-_h(f)\bigr\}$.

Otherwise, $f(u)-f(u')$ and $f(u)-f(u'')$ have the same sign.  Assuming that
they are both positive (the case where they are both negative can be treated
in a similar way), we have
\begin{align*}
d\bigl( f(u);[f(u'), f(u'')] \bigr) &= \min \bigl\{ f(u)-f(u')\,,\,
f(u)-f(u'') \bigr\}\\*
&= \min \bigl\{ \bigl(f(u)-f(u')\bigr)_+, \bigl(f(u'')-f(u)\bigr)_- \bigr\}\\*
&\leq \min\bigl\{ \omega^+_h(f), \omega^-_h(f)\bigr\}.
\end{align*}

Therefore, in both cases $d\bigl( f(u);[f(u'), f(u'')] \bigr) \leq \min\bigl\{
\omega^+_h(f), \omega^-_h(f)\bigr\}$, which implies the desired result.
\end{proof}

Note that as an immediate corollary of this theorem, we can see that if $f$ is
a monotonic function, then $\omega_h^{(M_1)}(f)=0$ for all~$h>0$ (of course,
we can also see this directly from the definition of the
$\omega_h^{(M_1)}(f)$).

Note also that in the inequality~\eqref{ineg-moduleM1} of
Theorem~\ref{module-de-de/croissance-avec-M1}, one can have both an equality
and a strict inequality (it is not difficult to construct the corresponding
examples).

The previous results allow, in particular, to relate the $M_1$ modulus of
continuity of a function whose all the jumps are of the same sign to the
uniform modulus of continuity of its continuous part.  Note that we need to
assume the summability of the jumps, in order to guarantee the existence of
the continuous part.  More generally, the decomposition of a càdlàg function
$f$ into a continuous and a pure jump parts is possible if and only if its
jumps are (absolutely) summable:
\[
\sum_{t\in J_f}\bil|f(t)-f(t-)\bir| < +\infty,
\]
where $J_f$ is the (at most countable) set of discontinuity points of~$f$.
Note also that the continuous part is defined up to an additive constant.
However, the various moduli of continuity (as well as those of increase and
decrease) are well defined, since they depend only on the increments of the
function.

\begin{theorem}
\label{Ineg-M1-U}
Let\/ $f\in \ED\bigl([a,b]\bigr)$ be a function whose all the jumps are of the
same sign and are summable, and let\/ $h>0$.  Then we have
\[
\omega_h^{(M_1)}(f) \leq \omega_h^{(U)}(f_{\mathrm{c.}}),
\]
where\/ $f_{\mathrm{c.}}$ is the continuous part of\/~$f$.
\end{theorem}

\begin{proof}
Let $h>0$ and $f\in \ED\bigl([a,b]\bigr)$.

Suppose, at first, that all the jumps of $f$ are negative.  Then, for all
$u,v\in[a,b]$ such that~$u\leq v\leq u+h$, we have
\[
\bigl( f(v)-f(u) \bigr)_+ \leq \bigl( f_{\mathrm{c.}}(v)-f_{\mathrm{c.}}(u)
\bigr)_+,
\]
and so
\[
\omega^+_h(f)
\leq
\omega^+_h(f_{\mathrm{c.}}).
\]
Thus, using Proposition~\ref{module-de-de/croissance-avec-U} and
Theorem~\ref{module-de-de/croissance-avec-M1}, we get
\[
\omega_h^{(M_1)}(f) \leq \min\bigl\{ \omega^+_h(f), \omega^-_h(f)\bigr\} \leq
\omega^+_h(f) \leq \omega^+_h(f_{\mathrm{c.}})  \leq \max\bigl\{
\omega^+_h(f_{\mathrm{c.}}), \omega^-_h(f_{\mathrm{c.}})\bigr\} =
\omega^{(U)}_h(f_{\mathrm{c.}}).
\]

Finally, in the case where all the jumps of $f$ are positive, we can
upper-bound $\omega^-_h(f)$ in a similar way and conclude using the same
argument.
\end{proof}

\subsection{Lipschitz growth and decay}

Now we introduce the notions of Lipschitz growth and decay and show how they
are related to the moduli of increase and decrease.
\begin{definition}
Let\/ $f$ be a function on an interval\/ $[a,b]\subset\RR$.
\begin{enumerate}
\item We say that\/ $f$ is\/ \emph{of Lipschitz growth} (or is\/
  \emph{Lipschitz increasing}), if there exists a constant\/ $L\geq 0$ such
  that
\[
\bigl( f(v)-f(u) \bigr)_+\leq L(v-u)
\]
for all\/ $u,v\in [a,b]$ such that\/ $u\leq v$.  In this case, we also say
that\/ $f$ is\/ \emph{of $L$--Lipschitz growth} (or is\/ \emph{$L$--Lipschitz
increasing}).
\item We say that\/ $f$ is\/ \emph{of Lipschitz decay} (or is\/
  \emph{Lipschitz decreasing}), if\/ $-f$ is of Lipschitz growth, that is, if
  there exists a constant\/ $L\geq 0$ such that
\[
\bigl( f(v)-f(u) \bigr)_-\leq L(v-u)
\]
for all\/ $u,v\in [a,b]$ such that\/ $u\leq v$.  In this case, we also say
that\/ $f$ is\/ \emph{of $L$--Lipschitz decay} (or is\/ \emph{$L$--Lipschitz
decreasing}).
\end{enumerate}
\end{definition}

Note that in the definition of the $L$--Lipschitz growth (resp.~decay), we
could equivalently have used $f(v)-f(u)$ (resp.~$f(u)-f(v)$) in the left hand
side of the inequalities, or simply have required the ratio
\[
\frac{f(v)-f(u)}{v-u}
\]
to be upper-bounded by $L$ (resp.~lower-bounded by~$-L$) for all $u,v\in
[a,b]$ such that~$u\ne v$.

Note also that if a function $f$ is $L$--Lipschitz continuous (with
some~$L\geq 0$), then it is obviously both $L$--Lipschitz increasing and
decreasing.  Inversely, if $f$ is both $L_1$--Lipschitz increasing and
$L_2$--Lipschitz decreasing (with~$L_1,L_2\geq 0$), then it is clearly
$L$--Lipschitz continuous with~$L=\max\{L_1,L_2\}$.

The following theorem establishes a relation between the Lipschitz growth
(resp.~decay) and the modulus of increase (resp.~decrease).
\begin{theorem}
\label{Module-crois-Lipsch}
Let\/ $f$ be a function on an interval\/ $[ a,b]\subset\RR$, and let\/ $L\geq
0$.
\begin{enumerate}
\item The function\/ $f$ is of\/ $L$--Lipschitz growth if and only if\/
  $\omega^+_h(f)\leq L h$ for all\/ $h>0$.
\item The function\/ $f$ is of\/ $L$--Lipschitz decay if and only if\/
  $\omega^-_h(f)\leq L h$ for all\/ $h>0$.
\end{enumerate}
\end{theorem}

\begin{proof}
It is sufficient to show only the first assertion of the theorem (the second
one can then be deduced by considering the function~$-f$).

Let $f$ be an $L$--Lipschitz increasing function on $[a,b]$, and let $h>0$.
Then, for all $u,v\in[a,b]$ such that $u\leq v\leq u+h$, we have
\[
\bigl( f(v)-f(u) \bigr)_+\leq L(v-u)\leq L h,
\]
and hence
\[
\omega^+_h(f) \leq L h.
\]

Let now $f$ be a function on $[a,b]$ such that $\omega^+_h(f) \leq L h$ for
all~$h>0$.  Then, for all $u,v\in[a,b]$ such that $u\leq v$, we have
\[
\bigl( f(v)-f(u) \bigr)_+\leq \omega^+_{v-u}(f) \leq L(v-u),
\]
which concludes the proof.
\end{proof}

Let us recall that a function having a bounded derivative is necessarily
Lipschitz continuous.  The following theorem shows that a differentiable
function whose derivative is upper- (resp.~lower-) bounded is necessarily
Lipschitz increasing (resp.~decreasing).
\begin{theorem}
Let\/ $L\geq 0$, $[a,b]\subset\RR$ and\/ $f\in\EC\bigl([a,b]\bigr)$.  Assume
further that\/ $f$ is differentiable on\/~$(a,b)$.
\begin{enumerate}
\item The function\/ $f$ is\/ $L$--Lipschitz increasing if and only if\/
  $f'(t)\leq L$ for all\/~$t\in (a,b)$.
\item The function\/ $f$ is\/ $L$--Lipschitz decreasing if and only if\/
  $f'(t)\geq -L$ for all\/~$t\in (a,b)$.
\end{enumerate}
\end{theorem}

\begin{proof}
It is again sufficient to show only the first assertion of the theorem (the
second one can then be deduced by considering the function~$-f$).

First we assume that $f'(t)\leq L$ for all~$t\in (a,b)$.  Then, by the mean
value theorem, for all $u,v\in[a,b]$ such that $u<v$, there exists $x\in
(u,v)$ verifying
\[
f(v)-f(u) = f'(x)( v-u ),
\]
and hence
\[
f(v)-f(u) \leq L( v-u ),
\]
that is, $f$ is $L$--Lipschitz increasing.

Now suppose that there exists $u\in (a,b)$ such that $f'(u)> L$.  Then, in the
vicinity of~$u$, we can find a point $v$ such that
\[
\frac{f(v)-f(u)}{v-u}> L,
\]
which contradicts the fact that $f$ is $L$--Lipschitz increasing.
\end{proof}

In fact, this theorem can be generalized to continuous piecewise
differentiable functions.  More precisely, it can be generalized to continuous
functions on $[a,b]$ which are differentiable everywhere on~$(a,b)$, except at
a finite number of points or, in other words, at functions
$f\in\EC\bigl([a,b]\bigr)$ such that $\Card\bigl((a,b) \setminus
D_f\bigr)<+\infty$, where we have denoted
\[
D_f=\{u\in (a,b) \: : \: f\text{ is differentiable at }u\}
\]
the set of differentiability points of~$f$.

\begin{theorem}
\label{Lipsch-Deri-Bor}
Let\/ $L\geq 0$, $[a,b]\subset\RR$ and\/ $f\in\EC\bigl([a,b]\bigr)$ such
that\/ $\Card\bigl((a,b) \setminus D_f\bigr)<+\infty$.
\begin{enumerate}
\item The function\/ $f$ is\/ $L$--Lipschitz increasing if and only if\/
  $f'(t)\leq L$ for all\/~$t\in D_f$.
\item The function\/ $f$ is\/ $L$--Lipschitz decreasing if and only if\/
  $f'(t)\geq -L$ for all\/~$t\in D_f$.
\end{enumerate}
\end{theorem}

\begin{proof}
It is still sufficient to show only the first assertion of the theorem (the
second one can then be deduced by considering the function~$-f$).  Moreover,
since the proof of the necessity carried out in the previous theorem remains
valid in this case, we need to show only the sufficiency.

So, let us suppose that $f'(t)\leq L$ for all $t\in D_f$, fix $u,v\in [ a,b ]$
with $u<v$, and denote~$(a,b) \setminus D_f = \{a_1,\ldots, a_m\}$, where
$a_0=a<a_1<\cdots<a_m<b=a_{m+1}$.

If there exists $i \in \{0, \ldots, m\}$ such that $u,v\in [ a_i,a_{i+1} ]$,
then, thanks to the argument used in the proof of the previous theorem, we
have
\[
f(v)-f(u) \leq L( v-u ).
\]

Let us now consider the case where there exists $i \in \{1, \ldots, m\}$ such
that $u\in [ a_{i-1},a_i )$ and~$v\in ( a_i,a_{i+1} ]$.  Then, applying the
mean value theorem on the intervals $[ u,a_i ]$ and on~$[ a_i,v ]$, we get
again
\[
f(v)-f(u) = \bigl(f(v)-f(a_i)\bigr) + \bigl(f(a_i)-f(u)\bigr) \leq
L(v-a_i)+L(a_i-u) = L(u-v).
\]

This argument is easily extended to the case where there are several
non-differentiability points of $f$ between $u$ and~$v$.  Indeed, let us
suppose that there exist $i,j \in \{0, \ldots, m\}$ with~$j-i>1$, such that
$u\in[a_i,a_{i+1} )$ and~$v\in( a_j,a_{j+1} ]$.  Then, applying the mean value
theorem on each of the intervals $[ u,a_{i+1} ], [ a_{i+1},a_{i+2} ], \ldots,
[ a_{j-1},a_{j} ], [ a_j,v ]$, we still get
\begin{align*}
f(v)-f(u) &= \bigl(f(v)-f(a_j)\bigr) + \bigl(f(a_j)-f(a_{j-1})\bigr)\\*
&\phantom{={}} + \cdots + \bigl(f(a_{i+2})-f(a_{i+1})\bigr) +
\bigl(f(a_{i+1})-f(u)\bigr)\\*
&\leq L(v-a_j)+L(a_j-a_{j-1})+\cdots + L(a_{i+2}-a_{i+1})+L(a_{i+1}-u)\\*
&= L(v-u).
\end{align*}

Thus, $f$ is $L$--Lipschitz increasing.
\end{proof}

Note that this theorem is not valid if instead of continuous piecewise
differentiable functions we consider piecewise differentiable càdlàg 
functions.  For example, the function $f(u)=\ind_{\{u\geq 1\}}$, $u\in[0,2]$,
is differentiable on $D_f=(0,2)\setminus\{1\}$ with $f'(u)=0$.  Hence, for any
$L\geq 0$, the derivative of $f$ is lower-bounded by $-L$ and upper-bounded by
$L$, while though $f$ is $L$--Lipschitz decreasing, it is not $L$--Lipschitz
increasing.  However, if we consider that at the point $1$ this function has a
derivative equal to $+\infty$, then the later will still be lower-bounded, but
no longer and upper-bounded, and so the theorem will ``remain true''.

Therefore, a natural way to generalize the previous theorem to piecewise
differentiable càdlàg functions is to use the following extension of the
notion of the derivative.  For a function~$f\in \ED\bigl([a,b]\bigr)$, we
denote $\widetilde{D}_f$ be the set of points $u\in (a,b)$ such that the ratio
\[
\frac{f(u+\delta)-f(u-\delta')}{\delta+\delta'}
\]
has a limit in~$\RR\cup\{-\infty,+\infty\}$ when $\delta \searrow 0$ and
$\delta' \searrow 0$ simultaneously.  For $u\in \widetilde{D}_f$, we call this
limit \emph{extended derivative\/} of $f$ at the point $u$, and we denote it
$f^{\circ}(u)$.  Note that $f^{\circ}(u)=f'(u)$ if $f$ is differentiable
in~$u$, and that $f^{\circ}(u)=+\infty$ (resp.~$f^{\circ}(u)=-\infty$) if $f$
admits a positive (resp.~negative) jump at the point~$u$.  Hence, the set
$\widetilde{D}_f$ contains the set~$D_f$, as well as the set $J_f$ of
discontinuity points of~$f$, but it may also contain some other points (such
as, for example, the points in the vicinity of which the function~$f$ behaves
like the function $u\mapsto \sign(u) \sqrt{\bil|u\bir|}$ does near~$u=0$).  It
should be noted that in order for the following theorem to be valid, it would
have been sufficient to extend the derivative to the set $D_f\cup J_f$, but we
find the above proposed extension more elegant.

\begin{theorem}
Let\/ $L\geq 0$, $[a,b]\subset\RR$ and\/ $f\in\ED\bigl([a,b]\bigr)$ such
that\/ $\Card\bigl((a,b) \setminus D_f\bigr)<+\infty$.
\begin{enumerate}
\item The function\/ $f$ is\/ $L$--Lipschitz increasing if and only if\/
  $f^{\circ}(t)\leq L$ for all\/~$t \in \widetilde{D}_f$.
\item The function\/ $f$ is\/ $L$--Lipschitz decreasing if and only if\/
  $f^{\circ}(t)\geq -L$ for all\/~$t\in \widetilde{D}_f$.
\end{enumerate}
\end{theorem}

\begin{proof}
It is once more sufficient to show only the first assertion of the theorem
(the second one can then be deduced by considering the function~$-f$).

The proof of the sufficiency follows the same pattern as that of the previous
theorem.  We suppose that $f^\circ(t)\leq L$ for all $t\in \widetilde{D}_f$
and, as before, we fix $u,v\in [ a,b ]$ with $u<v$, and denote $(a,b)
\setminus D_f = \{a_1,\ldots, a_m\}$, where $a_0=a<a_1<\cdots<a_m<b=a_{m+1}$.

We only consider the case where there is exactly one non-differentiability
point of $f$ between $u$ and~$v$ (the case where there are several such points
can be adapted similarly, and the case where there is none is the same as
before).  We therefore assume that there exists $i \in \{1, \ldots, m\}$ such
that $u\in [ a_{i-1},a_i )$ and $v\in ( a_i,a_{i+1} ]$

Note that as $f^\circ(t)\leq L<+\infty$, the jumps of the function $f$ are all
negative, and hence, in particular, $f(a_i)-f(a_i-)\leq 0$.  Introduce
equally, for $j\in\{i-1,i\}$, the function
\[
f_j(u) = \begin{cases}
f(u), &\text{if } u\in \bil[a_j,a_{j+1}\bir[,\\
f(a_{i+1}-), &\text{if } u=a_{j+1},
\end{cases}
\]
which is, in fact, the continuous extension on $[a_i,a_{i+1}]$ of the
restriction of $f$ on~$(a_i,a_{i+1})$.

Thus, applying the mean value theorem to the functions $f_{i-1}$ and $f_i$ on
the intervals $[ u,a_i ]$ and $[ a_i,v ]$ respectively, we finally get
\begin{align*}
f(v)-f(u) &= \bigl(f(v)-f(a_i)\bigr) + \bigl(f(a_i)-f(a_i-)\bigr) +
\bigl(f(a_i-)-f(u)\bigr)\\*
&\leq \bigl(f(v)-f(a_i)\bigr) + \bigl(f(a_i-)-f(u)\bigr)\\
&= \bigl(f_i(v)-f_i(a_i)\bigr) + \bigl(f_{i-1}(a_i)-f_{i-1}(u)\bigr)\\
&\leq L(v-a_i) + L(a_i -u)\\*
&= L(v-u).
\end{align*}

Now, for the proof of the necessity, let us suppose that there exists $u\in
\widetilde{D}_f$ such that~$f^\circ(u)> L$.  Then, in the vicinity of~$u$, we
can find two points $u$ and $u'$ such that
\[
\frac{f(u')-f(u'')}{u'-u''}> L,
\]
which contradicts the fact that $f$ is $L$--Lipschitz increasing.
\end{proof}

To conclude, let us note that the results of this section can be easily
adapted to the restricted moduli of continuity
\[
\Delta^{(M_1)}_h(f;[A, B])=\sup_{u\in [A,B]} \ \ \sup_{u',u''\in \RR \: : \:
  u-h\leq u'\leq u\leq u''\leq u+h} d\bigl( f(u);[f(u'), f(u'')] \bigr)
\]
and
\[
\Delta^{(U)}_h(f;[A, B])=\sup_{u\in [A,B]} \ \ \sup_{v\in \RR \: : \: u-h\leq
  v\leq u+h} \bil|f(u)-f(v)\bir|,
\]
using the restricted moduli of increase and decrease
\[
\Delta^\pm_h(f;[A, B])=\sup \bigl( f(v)-f(u) \bigr)_\pm \: ,
\]
where the $\sup$ is now taken over all $u,v\in\RR$ such that (at least) one of
them belongs to~$[A,B]$ and~$u\leq v\leq u+h$.  The only delicate point is the
relation between the Lipschitz growth (resp.~decay) and the modulus of
increase (resp.~decrease) established in Theorem~\ref{Module-crois-Lipsch}.
For example, in order to show the inequality $\Delta^{+}_h(f;[A, B]) \leq L
h$, we need to suppose that $f$ is $L$--Lipschitz increasing on the interval
$[A-h,B+h]\supsetneq [A,B]$, while if we suppose that this inequality holds,
we can only show that $f$ is $L$--Lipschitz increasing on the
interval~$[A,B]$.

\section{Application to the smooth change-point model for Poisson processes}
\label{Sec-EMV-rapide}

We start this section by recalling the smooth change-point model for Poisson
processes considered by the authors in \citet{Amiri-Dachian}.  Let
$n\in\NN^*$, let $0<\alpha < \beta <\tau$ be some known constants, and let
$\psi$ be some known strictly positive continuous function on $[0,\tau]$.  Let
also $r>-\min_{0\leq t\leq \tau} \psi (t)$ be some known constant,
$({\delta_n})_{n\in\NN}$ be some known sequence decreasing to~$0$, and let
$\theta \in \Theta =(\alpha,\beta)$ be a one-dimensional unknown parameter.
The aim is to estimate $\theta$ (in the asymptotics~$n\to+\infty$) from the
observation $X^{(n)}=(X_1,\ldots, X_n)$, where $X_j=\bigl(X_j(t),\ 0 \leq t
\leq \tau \bigr)$, $j = 1,\ldots , n$, are independent inhomogeneous Poisson
processes on the interval $[0,\tau]$ with intensity function
$\lambda_{\theta}=\lambda_{\theta}^{(n)}$ given by
\[
\lambda_{\theta}^{(n)}(t) = \psi(t) + \frac{r}{\delta_n} \, (t-\theta) \,
\ind_{[\theta, \theta+\delta_n)}(t) + r \, \ind_{[\theta+\delta_n,
      \tau]}(t),\qquad 0\leq t\leq \tau.
\]
This function, in an important particular case $\psi \equiv\lambda_0>0$, is
presented in Figure~\ref{fig-lambda}.

\begin{figure}[!ht]
\centering
\includegraphics[scale=0.35]{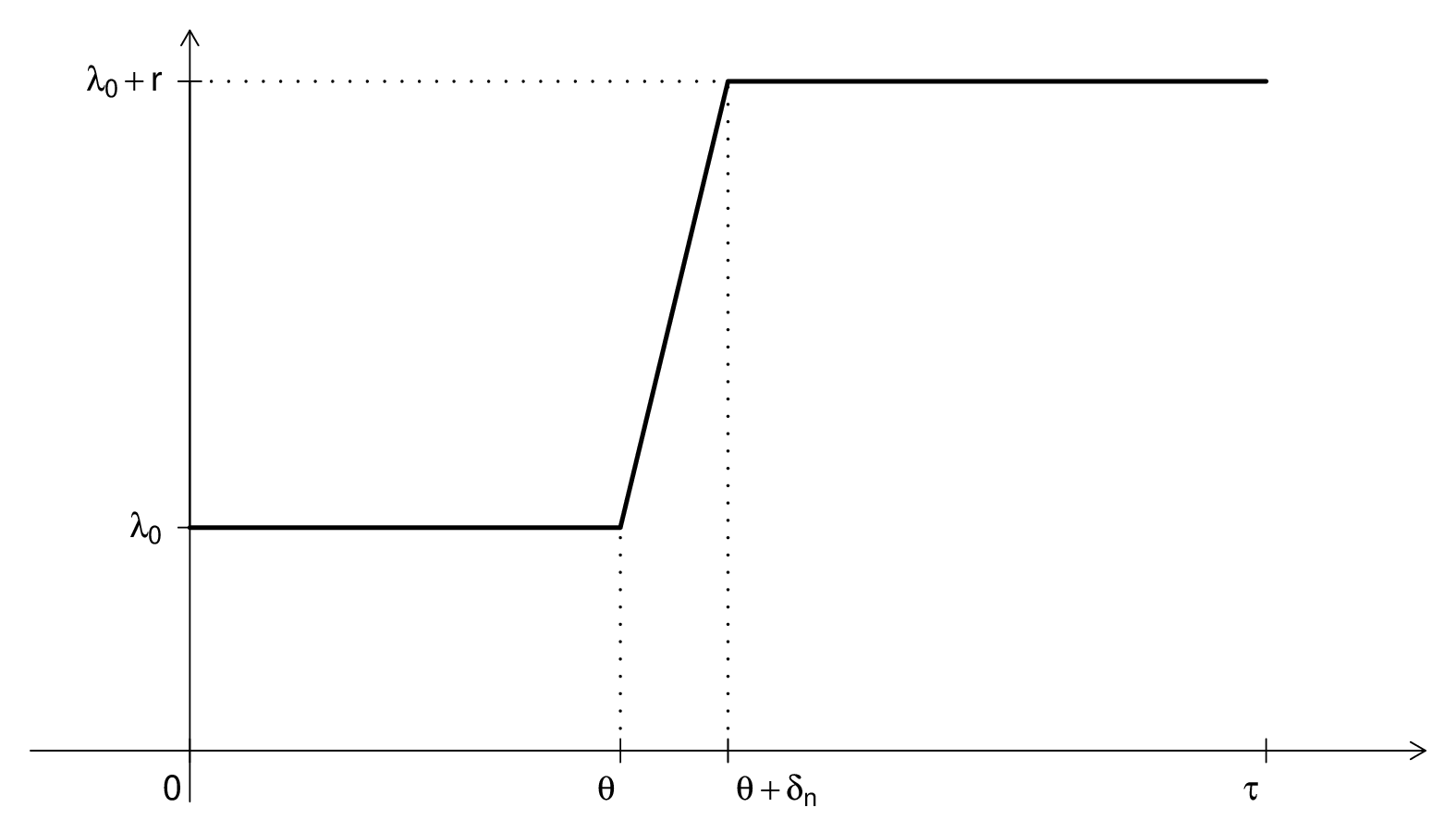}
\caption{intensity function $\lambda_\theta^{(n)}$ with $\psi\equiv\lambda_0$}
\label{fig-lambda}
\end{figure}

Note that this model of observation is equivalent to observing a single
realization on the interval~$[0, n\tau]$ of an inhomogeneous Poisson process
$X=\bigl(X(t),\ 0\leq t\leq n\tau\bigr)$ with a~$\tau$-periodic intensity
function equal to $\lambda_\theta^{(n)}$ on the first period (large
observation time asymptotics).  Also, it is equivalent to observing a single
realization on the interval~$[0, \tau]$ of an inhomogeneous Poisson process
$Y^{(n)}=\bigl(Y^{(n)}(t),\ 0\leq t\leq \tau \bigr)$ of intensity function
$\Lambda_{\theta}^{(n)} = n \, \lambda_{\theta}^{(n)}$ (large intensity
asymptotics).

We denote $\Pb_\theta = \Pb_\theta^{(n)}$ the probability measure
corresponding to $X^{(n)}$.  We also denote $\Ex_\theta = \Ex_\theta^{(n)}$
the corresponding mathematical expectation.  The likelihood, with respect to
the measure $\Pb_{*} = \Pb_{*}^{(n)}$ corresponding to $n$ independent
homogeneous Poisson processes of unit intensity, is given (see, for example,
\citet{LiShi01}) by
\[
L \bigr(\theta, X^{(n)}\bigl) =
\frac{\dd\Pb_{\theta}\bigl(X^{(n)}\bigr)}{\dd\Pb_{*}} = \exp\Biggl\{
\sum_{j=1}^{n} \int_{0}^{\tau} \ln\bigl(\lambda_{\theta}(t)\bigr)\dd
X_{j}(t)-n\int_{0}^{\tau} \bigl(\lambda_{\theta}(t)-1 \bigr) \dd t\Biggr\},
\qquad\!\! \theta \in \Theta .
\]

As estimators of the unknown parameter $\theta$, we consider the maximum
likelihood estimator~(MLE) and the Bayesian estimators~(BEs).  The
MLE~$\hat{\theta}_n$ is given by
\[
\hat{\theta}_n = \argsup_{\theta \in \Theta} L \bigl(\theta, X^{(n)} \bigr),
\]
and the BE $\tilde{\theta}_n$ for quadratic loss and prior density $q$ is
given by
\[
\tilde{\theta}_n = \frac{\int_{\alpha}^{\beta}\theta \, q(\theta) \, L
  \bigl(\theta, X^{(n)} \bigr) \dd \theta}{\int_{\alpha}^{\beta} q(\theta) \,
  L \bigl(\theta, X^{(n)} \bigr) \dd \theta} \, .
\]

The study of the asymptotic behavior of the MLE and of the BEs using the
likelihood ratio analysis method was initiated by the authors in
\citet{Amiri-Dachian}.  Note that the normalized likelihood ratio process of
the model has the form
\[
Z_n(u) = \exp\Biggl\{ \sum_{j=1}^{n} \int_{0}^{\tau}
\ln\biggl(\frac{\lambda_{\theta+u\varphi_n}(t)}{\lambda_{\theta}(t)}
\biggr)\dd X_{j}(t) - n\int_{0}^{\tau} \bigl(\lambda_{\theta+u\varphi_n}(t) -
\lambda_{\theta}(t) \bigr) \dd t\Biggr\},\qquad u\in\UU_n,
\]
and that its trajectories (after extension to the whole real line) almost
surely belong to the space $\EC_0(\RR)$.  It turns out that the asymptotic
behavior of the likelihood ratio of the model (as well as the properties of
the MLE and of the BEs) depends on the rate of convergence of~$\delta_n$ to
zero.  More precisely, there are three different cases:
\begin{align}
&n\delta_n\xrightarrow[n\to +\infty]{} +\infty,\label{casl}\\*
&n\delta_n\xrightarrow[n\to +\infty]{} 0\label{casr}\\
\intertext{and}
&n\delta_n\xrightarrow[n\to +\infty]{} c>0.\label{casc}
\end{align}

The asymptotic properties of the MLE and of the BEs in the case~\eqref{casl},
which we call \emph{slow regime}, as well as those of the BEs in the
case~\eqref{casr}, which we call \emph{fast regime}, were established in
\citet{Amiri-Dachian}.

In the slow regime, it was shown that the model is locally asymptotically
normal (LAN) with a likelihood normalization rate
$\varphi_n=\sqrt{\delta_n/n}\,$, and that the MLE and the BEs are consistent,
asymptotically normal and asymptotically efficient (convergence of moments of
the estimators also holds).  Note that here the trajectories of the limiting
likelihood ratio process are also continuous, and the weak convergence of the
likelihood ratio takes place in the space~$\EC_0^{(U)}(\RR)$.  Note also, that
except the unusual rate $\sqrt{\delta_n/n}\,$, the behavior of the model and
of the estimators is similar to that of the regular case studied by
\citeauthor{Kut98} in~\myciteyearp{Kut77,Kut79,Kut84,Kut98}.

In the fast regime, taking the likelihood normalization rate $\varphi_n$ to
be~$1/n$, it was shown that the normalized likelihood ratio process has the
following properties.
\begin{itemize}
\item The finite-dimensional distributions of the process $Z_n$ converge to
  those of the process $Z^\star_{a,b}$ with $a=\psi(\theta)$,
  $b=\psi(\theta)+r$ and
\[
Z^\star_{a,b}(u) = \begin{cases}
\vphantom{\bigg(} \exp \Bigl\{ \ln \bigl( \frac a b\bigr) Y^+(u)
+(b-a)u\Bigr\}, &\text{if }u\in\RR_+,\\
\vphantom{\bigg(} \exp \Bigl\{ \ln \bigl( \frac b a \bigr)
Y^-\bigl((-u)-\bigr) + (b-a)u\Bigr\}, &\text{if }u\in\RR_-,
\end{cases}
\]
where $Y^+$ and $Y^-$ are independent Poisson processes on~$\RR_+$ with
intensities $b$ and $a$ respectively.
\item There exists a constant $C>0$ such that
\[
\Ex_\theta \bigl\vert Z_n^{1/2} (u)- Z_n^{1/2} (v) \bigr\vert ^2 \leq C\bil| u
- v \bir|
\]
for all $n\in\NN^*$, $u,v \in \UU_n$ and $\theta \in \Theta$.
\item There exists a constant $\kappa>0$ such that for all $n$ sufficiently
  large, we have
\[
\Ex_\theta Z_n^{1/2}(u) \leq \exp\bigl\{-\kappa \min \{\bil| u \bir|, u^2\}
\bigr\}
\]
for all $u\in \UU_n$ and $\theta \in \Theta$.
\end{itemize}

These properties were sufficient to deduce the properties of the BEs in
\citet{Amiri-Dachian}.  Note that the process $Z^\star_{a,b}$ is the limiting
likelihood ratio process of the ``pure'' (discontinuous) change-point models
for Poisson processes studied by \citeauthor{Kut98}
in~\myciteyearp{Kut78,Kut84,Kut98}, and so the properties of the BEs are
exactly the same as in that case.  As to the MLE, its study was not possible
at that time, as the likelihood ratio analysis method was available only in
the spaces $\EC_0^{(U)}(\RR)$ and $\ED_0^{(J_1)}(\RR)$.  Indeed, here the
trajectories of the normalized likelihood ratio process are continuous, while
those of the limiting likelihood ratio process are discontinuous, and so, as
it was already mentioned, the weak convergence cannot take place in either of
these spaces.

Now, employing the $M_1$ version of the likelihood ratio analysis method
developed in this paper, we can establish the following properties of the MLE
in the fast regime, thereby completing the study of the latter.

\begin{theorem}
\label{EMV-CR}
Suppose $n\delta_n\to 0$.  Then, the MLE $\hat{\theta}_n$ has, uniformly with
respect to $\theta\in\KK$ for any compact $\KK\subset\Theta$, the following
properties:
\begin{itemize}
\item $\hat{\theta}_n$ is consistent;
\item $\hat{\theta}_n$ converges at rate $1/n$ and its limiting distribution
  is that of the random variable
\[
\eta_{a,b}=\argsup_{u\in \RR} Z^\star_{a,b}(u),
\]
that is,
\[
n \bigl(\hat{\theta}_n-\theta\bigr) \Longrightarrow \eta_{a,b}\; ;
\]
\item we have the convergence of polynomial moments, that is,
\[
\lim_{n\to +\infty} n^p \, \Ex_\theta \bigl|\hat{\theta}_n-\theta\bigr|^p =
\Ex {\bil| \eta_{a,b} \bir|}^p
\]
for any $p>0$.
\end{itemize}
\end{theorem}

As we have already said, we prove this theorem by using the $M_1$ version of
the likelihood ratio analysis method, that is, we establish the weak
convergence of $Z_n$ to $Z^\star_{a,b}$ in the space $\ED_0^{(M_1)}(\RR)$ and
then deduce the desired properties of the MLE.  For this, it is sufficient to
apply Theorem~\ref{Cv-M1-unif} (recall that the assertion~\ref{Res3} of the
latter theorem, together with the convergence of the finite-dimensional
distributions, guarantees the weak convergence of the likelihood ratios, from
which we can deduce the convergence and the limiting distribution of the
estimator, and that the assertion~\ref{Res2} allows to deduce the convergence
of its polynomial moments).

Before checking the conditions of Theorem~\ref{Cv-M1-unif}, let us note that
it shows that the properties of the MLE are also (like those of the BEs)
exactly the same as in the pure change-point models for Poisson processes.
Note also that this situation (same behavior of the estimator as in the pure
change-point case if the transition interval shrinks sufficiently fast, and
Gaussian limiting behavior with an unusual rate otherwise) is similar to what
happens in the smooth change-point model of signal observed in a white
Gaussian noise (cf.~Remark~5 of \citet{IbHas75}).  It should be also mentioned
that both the normalized and the limiting likelihood ratios being continuous
in the latter model, its study uses the space $\EC_0(\RR)$ and has not to deal
with topology issues on the space $\ED_0(\RR)$.

Let us also precise how we extend the process $Z_n(u)$ from the interval
$\UU_n$ to~$\RR$.  We could, of course, have simply set it equal to $0$
outside of $\UU_n$ (adjusting the values on the edges of $\UU_n$ in such a way
that the trajectories belong to the space $\ED_0(\RR)$).  But in order to
facilitate the study of the modulus of continuity we proceed differently.

First, we denote $u_n^-$ (resp.~$u_n^+$) the left (resp.~right) edge of the
interval~$\UU_n$, and extend the trajectories of $Z_n$ at these two points by
continuity.  Then, outside of~$\UU_n$, we decrease (as $u$ moves away
from~$\UU_n$) the trajectories of $\ln Z_n$ linearly with a slope equal
to~$\bil|r\bir|$.  Otherwise speaking, we extend $Z_n$ by putting $Z_n(u) =
Z_n(u_n^+) \exp{\{-\bil|r\bir|(u-u_n^+)\}}$ for~$u>u_n^+$, and $Z_n(u) =
Z_n(u_n^-) \exp{\{-\bil|r\bir|(u_n^--u)\}}$ for~$u<u_n^-$.  Note that the
trajectories of the obtained process belong to space~$\EC_0(\RR)\subset
\ED_0(\RR)$, and that we have the following lemma.

\begin{lemma}
The function $\ln Z_n(u)$, $u\in\RR$, is $r$--Lipschitz increasing in the case
$r>0$, and $\bil|r\bir|$--Lipschitz decreasing in the case $r<0$.
\end{lemma}

\begin{proof}
We consider only the case $r>0$ (the case $r<0$ can be treated in a similar
way).  We are going to prove this lemma by applying
Theorem~\ref{Lipsch-Deri-Bor}.  For this, note that for $u\in\UU_n$, we have
\begin{align*}
\ln Z_n(u) &= \sum_{j=1}^{n} \int_{0}^{\tau}
\ln\biggl(\frac{\lambda_{\theta+\frac{u}{n}}(t)}{\lambda_{\theta}(t)}\biggr)
\dd X_{j}(t) - n\int_{0}^{\tau} \bigl( \lambda_{\theta+\frac{u}{n}}(t) -
\lambda_{\theta}(t)\bigr) \dd t\\*
&= \sum_{j=1}^n\sum_{i=1}^{X_j(T)}
\ln\biggl(\frac{\lambda_{\theta+\frac{u}{n}}(t_{i,j})}{\lambda_{\theta}(t_{i,j})}\biggr)
+ ru,
\end{align*}
where $t_{i,j}$ is the $i$-th ``event'' (discontinuity~point) of~$X_j$.

Since
\[
\lambda_{\theta+\frac{u}{n}}(t_{i,j}) = \psi(t_{i,j}) + \frac{r}{\delta_n} \,
\Bigl( t_{i,j}-\theta-\frac{u}{n} \Bigr) \, \ind_{[\theta +\frac{u}{n},
    \theta+\frac{u}{n}+\delta_n)}(t_{i,j}) + r \,
  \ind_{[\theta+\frac{u}{n}+\delta_n, \tau]}(t_{i,j}),
\]
every term of the last sum is differentiable with respect to $u$ everywhere,
except (at most) at two points: $n(t_{i,j}-\theta)$ and
$n(t_{i,j}-\theta)-n\delta_n$ (depending if they belong or not to~$\UU_n$).
The function $\ln Z_n$ is therefore piecewise differentiable.

Moreover, for all $j=1,\ldots,n$, $i=1,\ldots,X_j(T)$ and $u\in\UU_n$ which is
a differentiability point of $\lambda_{\theta+\frac{u}{n}}(t_{i,j})$, we have
\[
\frac{\dd}{\dd u} \,
\ln\biggl(\frac{\lambda_{\theta+\frac{u}{n}}(t_{i,j})}{\lambda_{\theta}(t_{i,j})}\biggr)
= \frac{\dd}{\dd u} \, \ln \bigl(\lambda_{\theta+\frac{u}{n}}(t_{i,j})\bigr) =
\frac{1}{\lambda_{\theta+\frac{u}{n}}(t_{i,j})}\, \frac{\dd}{\dd u} \,
\lambda_{\theta+\frac{u}{n}}(t_{i,j}) \leq 0.
\]
Indeed, $\lambda_{\theta+\frac{u}{n}}(t_{i,j}) \geq 0$ and $\frac{\dd}{\dd u}
\, \lambda_{\theta+\frac{u}{n}}(t_{i,j}) \leq 0$, since
\[
\frac{\dd}{\dd u} \, \lambda_{\theta+\frac{u}{n}}(t_{i,j}) =
\begin{cases}
-\frac{r}{n\delta_n}\, , &\text{if } u\in(n(t_{i,j}-\theta)-n\delta_n,
n(t_{i,j}-\theta)) \cap \UU_n,\\
0, &\text{if } u\in[n(t_{i,j}-\theta)-n\delta_n,
  n(t_{i,j}-\theta)]^{\mathrm{c}} \cap \UU_n.
\end{cases}
\]

Thus,
\[
\frac{\dd}{\dd u}\, \ln Z_n(u) = \sum_{j=1}^n\sum_{i=1}^{X_j(T)}
\frac{\dd}{\dd u} \,
\ln\biggl(\frac{\lambda_{\theta+\frac{u}{n}}(t_{i,j})}{\lambda_{\theta}(t_{i,j})}\biggr)
+ r \leq r
\]
for every differentiability point $u\in\UU_n$ of $\ln Z_n$.

Finally, as we have $\frac{\dd}{\dd u}\, \ln Z_n(u)=-r$ for $u>u_n^+$ , and
$\frac{\dd}{\dd u}\, \ln Z_n(u)=r$ for $u<u_n^-$, we obtain $\frac{\dd}{\dd
  u}\, \ln Z_n(u) \leq r$ for every differentiability point $u\in\RR$ of~$\ln
Z_n$.  Therefore, according to Theorem~\ref{Lipsch-Deri-Bor}, the function
$\ln Z_n(u)$, $u\in\RR$, is $r$--Lipschitz increasing.
\end{proof}

It is interesting to note that this lemma will allow us to control the
restricted modulus of continuity of the function $\ln Z_n$ in the proof of the
following lemma.  Indeed, for example, in the case $r>0$, this lemma
guarantees that $\ln Z_n$ is $r$--Lipschitz increasing, and although it is
only $\frac{r}{n\delta_n}$--Lipschitz decreasing (recall that we are in the
fast regime, and that we have~$\frac{r}{n\delta_n}\gg r$), due to
Theorem~\ref{module-de-de/croissance-avec-M1}, the smaller of the two
constants prevails.

To verify the condition~\ref{CondA} of Theorem~\ref{Cv-M1-unif}, it is
sufficient, due to Remark~\ref{CondA'-Cv-M1}, to verify the
condition~\ref{CondA'}, that is, to prove the following lemma.

\begin{lemma}
Let $\KK\subset\Theta$ be a compact.  Then, there exist $\gamma>0$ and $h_0>0$
such that for all $0<h<h_0$, $n\in \NN^*$, $l\in \ZZ$ and $\theta \in \KK$, we
have
\begin{equation}
\label{ineg-M1-Zn}
\Pb_\theta\Bigl( \Delta^{(M_1)}_h\bigl(Z_n^{1/2};[l, l+1]\bigr)>h^\gamma
\Bigr)\leq C\, h^{2\gamma},
\end{equation}
where $C$ is a positive constant.
\end{lemma}

\begin{proof}
We consider again only the case $r>0$ (the case $r<0$ can be treated in a
similar way).  For all $u,u'\in[l,l+1]$, applying the mean value theorem to
the exponential function between $\ln Z_n^{1/2}(u)$ and $\ln Z_n^{1/2}(u')$,
we have
\begin{align*}
\bigl| Z_n^{1/2}(u)-Z_n^{1/2}(u')\bigr| &= \bigl| \exp \bigl\{ \ln
Z_n^{1/2}(u) \bigr\} - \exp \bigl\{ \ln Z_n^{1/2}(u') \bigr\} \bigr|\\*
&\leq \bigl| \ln Z_n^{1/2}(u) - \ln Z_n^{1/2}(u') \bigr|
\max\{Z_n^{1/2}(u),Z_n^{1/2}(u') \}\\*
&\leq \frac{1}{2}\, \bigl|\ln Z_n(u) - \ln Z_n(u') \bigr| \sup_{v\in [l,l+1]}
Z_n^{1/2}(v),
\end{align*}
which implies (noting that the function $x\mapsto\exp\{x/2\}$ is monotone)
that
\[
\Delta^{(M_1)}_h\bigl(Z_n^{1/2};[l, l+1]\bigr) \leq \frac{1}{2}\,
\Delta^{(M_1)}_h\bigl(\ln Z_n;[l, l+1]\bigr) \sup_{v\in [l,l+1]} Z_n^{1/2}(v).
\]
As the function $\ln Z_n$ is $r$--Lipschitz increasing, for all $u\in[l,l+1]$,
we have
\[
\ln Z_n(u) \leq \ln Z_n(l)+r,
\]
which yields
\[
Z_n^{1/2}(u) \leq e^{r/2}\, Z_n^{1/2}(l).
\]
Hence, we obtain
\[
\Delta^{(M_1)}_h\bigl(Z_n^{1/2};[l, l+1]\bigr) \leq \frac{1}{2}\,
\Delta^{(M_1)}_h\bigl(\ln Z_n;[l, l+1]\bigr) \, e^{r/2}\, Z_n^{1/2}(l).
\]

Using Theorems~\ref{module-de-de/croissance-avec-M1}
and~\ref{Module-crois-Lipsch}, we get
\begin{align*}
\Delta^{(M_1)}_h\bigl(\ln Z_n;[l, l+1]\bigr) &\leq \min\bigl\{
\Delta^+_h\bigl(\ln Z_n;[l, l+1]\bigr), \Delta^-_h\bigl(\ln Z_n;[l,
  l+1]\bigr)\bigr\}\\*
&\leq \Delta^+_h\bigl(\ln Z_n;[l, l+1]\bigr)\\*
&\leq r h,
\end{align*}
and we can thus write
\begin{align*}
\Pb_\theta\Bigl( \Delta^{(M_1)}_h\bigl(Z_n^{1/2};[l, l+1]\bigr)>h^\gamma
\Bigr) &\leq \Pb_\theta\biggl( \frac{1}{2}\, \Delta^{(M_1)}_h\bigl(\ln Z_n;[l,
  l+1]\bigr)\, e^{r/2}\, Z_n^{1/2}(l)>h^\gamma \biggr)\\*
&\leq \Pb_\theta\bigl( r h >h^{2\gamma} \bigr) + \Pb_\theta\biggl(
\frac{1}{2}\, e^{r/2}\, Z_n^{1/2}(l) >h^{-\gamma} \biggr).
\end{align*}

For the first term, taking $\gamma<1/2$ and putting
$h_0=r^{\frac{1}{2\gamma-1}}$, we obtain
\[
\Pb_\theta\bigl( r h > h^{2\gamma} \bigr) = \Pb_\theta\bigl( r > h^{2\gamma-1}
\bigr) = \Pb_\theta( h_0 < h ) = 0
\]
for all $0<h<h_0$.

For the second term, by Markov's inequality we have
\[
\Pb_\theta\biggl( \frac{1}{2}\, e^{r/2}\, Z_n^{1/2}(l) >h^{-\gamma} \biggr) =
\Pb_\theta\Bigl( Z_n(l) > 4e^{-r}h^{-2\gamma} \Bigr) \leq \frac{\Ex_\theta
  Z_n(l)}{4e^{-r}h^{-2\gamma}}\, ,
\]
and as $\Ex_\theta Z_n(l)\leq 1$ (actually, we have $\Ex_\theta Z_n(l)= 1$ if
$l\in\UU_n$, and we can bound $\Ex_\theta Z_n(l)$ by $\Ex_\theta Z_n(u_n^+)$
or by $\Ex_\theta Z_n(u_n^-)$ otherwise), we get
\[
\Pb_\theta\biggl( \frac{1}{2}\, e^{r/2}\, Z_n^{1/2}(l) >h^{-\gamma} \biggr)
\leq \frac{e^r}{4}\, h^{2\gamma}.
\]

So,
\[
\Pb_\theta\Bigl( \Delta^{(M_1)}_h\bigl(Z_n^{1/2};[l, l+1]\bigr)>h^\gamma
\Bigr) \leq \frac{e^r}{4}\, h^{2\gamma},
\]
and the inequality~\eqref{ineg-M1-Zn} holds (with $C=e^r/4$).
\end{proof}

It remains to verify the condition~\ref{CondB} of Theorem~\ref{Cv-M1-unif}.
More precisely, taking into account Remarks~\ref{CondB'-Cv-M1}
and~\ref{n-C1-Cv-M1}, we will instead show the following lemma.

\begin{lemma}
Let $\KK\subset\Theta$ be a compact.  Then, there exist $\kappa'>0$ and
$n_0\in\NN^*$ such that
\[
\Ex_\theta Z_n^{1/2}(u) \leq \exp\bigl\{-\kappa' \min \{\bil| u \bir|, u^2\}
\bigr\}
\]
for all $n\geq n_0$, $u\in \RR$ and $\theta \in \KK$.
\end{lemma}

\begin{proof}
First of all, recall that for $u\in \UU_n$ the desired inequality was already
obtained (and is valid for any~$\kappa'\leq \kappa$).

Now, if $u\geq u_n^+$ , putting $\kappa'=\min\{ \kappa,\bil|r\bir|/2 \}$ we
get
\begin{align*}
\Ex_\theta Z_n^{1/2}(u) &= \Ex_\theta Z_n^{1/2}(u_n^+)
\exp{\{-\bil|r\bir|(u-u_n^+)/2\}}\\*
&\leq \exp\{-\kappa u_n^+ \} \exp{\{-\bil|r\bir|(u-u_n^+)/2\}}\\
&\leq \exp\{-\kappa' u_n^+ \} \exp{\{-\kappa'(u-u_n^+)\}}\\*
&= \exp{\{-\kappa'u\}}.
\end{align*}
Here, without loss of generality, we have assumed~$u_n^+\geq1$ (indeed, it is
enough to adjust~$n_0$).

The case $u\leq u_n^-$ can be treated in a similar way, and so the inequality
is established for all~$u\in\RR$.
\end{proof}

Thus, all the conditions of Theorem~\ref{Cv-M1-unif} are satisfied, and hence
the (uniform with respect to~$\theta \in \KK$) weak convergence of $Z_n$ to
$Z^\star_{a,b}$ in the space $\ED_0^{(M_1)}(\RR)$ and the properties of the
MLE stated in Theorem~\ref{EMV-CR} follow.

\section{Conclusion and discussion}
\label{Sec-Discussion}

For models with discontinuous likelihood ratio processes, the
Ibragimov--Khasminskii likelihood ratio analysis method was, prior to this
work, developed and applied using the usual Skorokhod topology~$J_1$.  In this
paper, we propose a version of the likelihood ratio analysis method which uses
the (weaker) Skorokhod $M_1$ topology.  We then apply this $M_1$ version of
the likelihood ratio analysis to the smooth change-point model for Poisson
processes.

The latter model is interesting on its own.  Its study was initiated by the
authors in \citet{Amiri-Dachian}.  It was shown there that in the slow
regime~\eqref{casl}, the behavior of the model and of the estimators (MLE and
BEs) is similar to the regular case (though with an unusual normalization
rate~$\varphi_n=\sqrt{\delta_n/n}\,$), and that in the fast
regime~\eqref{casr}, the behavior of the BEs is exactly the same as in the
pure change-point case.  In this paper, employing the newly developed $M_1$
version of the likelihood ratio analysis, we were able to show that the
behavior of the MLE is also exactly the same as in the pure change-point case.
We believe that these results provide a theoretical explanation of why pure
change-point models have proved successful in real-world applications, despite
the fact that physical systems cannot switch instantaneously from one regime
to another through a discontinuous transition.  Indeed, it is more realistic
to suppose that the transition happens smoothly over a small interval, and if
this interval is short enough (which is often the case in practice), the
asymptotic behavior is the same.  Note that the \emph{critical
regime\/}~\eqref{casc} has not yet been considered, and its study is one of
the possible avenues for this work.

Another natural direction is to use the $M_1$ version of the likelihood ratio
analysis to investigate other models whose normalized likelihood ratios are
continuous, while their limiting counterparts are discontinuous.  Remaining in
the framework of Poisson processes, one can consider, for example, intensities
of the form
\[
\psi(t) + r \, \biggl(\frac{t-\theta}{\delta}\biggr)^\kappa \, \ind_{[\theta,
    \theta+\delta)}(t) + r \, \ind_{[\theta+\delta, \tau]}(t),\qquad 0\leq
  t\leq \tau,
\]
with either $\delta=\delta_n\to 0$, or $\kappa=\kappa_n\to 0$ (or both).
Also, it would be interesting to apply these techniques to other observation
models, the i.i.d.\ observation model being a natural starting point.

Finally, we believe that the $M_1$ version of the likelihood ratio analysis
can also facilitate the study of models where both the normalized and the
limiting likelihood ratios are discontinuous (where, of course, the classical
$J_1$ version also works).  Indeed, since the topology~$M_1$ is weaker
than~$J_1$, one can expect the corresponding tightness criteria, and in
particular the control of the $M_1$ modulus of continuity, to be easier to
verify.  For example, for the study of the i.i.d.\ and Poissonian pure
change-point models, \citeauthor{IbHas81}
in~\myciteyearp{IbHas70,IbHas72,IbHas81} and \citeauthor{Kut98}
in~\myciteyearp{Kut78,Kut84,Kut98} use essentially the same technique to
control the $J_1$ modulus of continuity.  They need to control both the
increments of the continuous part $Z_{n,\mathrm{c.}}^{1/2}$ of $Z_n^{1/2}$
through an inequality of the form
\begin{equation}
\label{ctrl-cp}
\Ex_\theta \bigl| Z_{n,\mathrm{c.}}^{1/2} (u)- Z_{n,\mathrm{c.}}^{1/2} (v)
\bigr| ^2 \leq C\bil| u - v \bir|^2,
\end{equation}
and the probabilities of having one or two jumps on a small interval through
inequalities of the form
\begin{equation}
\label{ctrl-j}
\Pb_\theta (A_i)\leq C_i h^i,\qquad i=1,2,
\end{equation}
where $A_i=A_i^{(n)}(u,u+h)$ is the event $\bigl\{Z_n$ admits at least $i$
discontinuity points on the interval $(u,u+h)\bigr\}$.  Consider, for
instance, the case where exactly one change-point (one discontinuity) is
present in the initial model.  Then all the jumps of $Z_n$ are of the same
sign and, according to Theorem~\ref{Ineg-M1-U}, the control of the $M_1$
modulus of continuity only requires establishing the
inequality~\eqref{ctrl-cp}.  Therefore, there is no need to derive the
inequalities~\eqref{ctrl-j}, which substantially shortens the proofs.

\bigskip
\noindent\textbf{Acknowledgments.}
The authors acknowledge the support of the CDP C$^2$EMPI, together with the
French State under the France-2030 program, the University of Lille, the
Initiative of Excellence of the University of Lille, the European Metropolis
of Lille for their funding and support of the R-CDP-24-004-C2EMPI project.

\bibliographystyle{abbrvnat}
\bibliography{bibSCPFC.bib}

\end{document}